\documentclass[10pt]{amsart}
\usepackage[USenglish]{babel}
\usepackage{amsmath, amsfonts, amssymb, amsthm, amscd}
\usepackage{bbm, color, csquotes, dsfont, enumitem, epsf, graphicx, indentfirst, latexsym, mathtools, rotating, scalerel, tikz, tikz-cd, verbatim}
\usepackage[final, breaklinks=true, colorlinks, citecolor=green, linkcolor=blue, pdfusetitle]{hyperref}

\usepackage[T1]{fontenc}

\usepackage[abbrev]{amsrefs}

\usepackage{lmodern}

\newcommand{\bm}[1]{\boldsymbol{#1}}

\newtheorem{theorem}{Theorem}[section]

\newtheorem{lemma}[theorem]{Lemma}
\newtheorem{proposition}[theorem]{Proposition}
\newtheorem{corollary}[theorem]{Corollary}

\theoremstyle{definition}
\newtheorem{definition}[theorem]{Definition}
\newtheorem{notation}[theorem]{Notation}

\theoremstyle{remark}
\newtheorem{remark}[theorem]{Remark}

\numberwithin{equation}{section}

\newcommand{\R}{\mathbb{R}}
\newcommand{\Z}{\mathbb{Z}}

\newcommand{\Q}{\mathbb{Q}}

\newcommand{\ep}{\varepsilon}

\newcommand{\g}{\gamma}

\newcommand{\cB}{{\mathcal B}}

\newcommand{\cM}{{\mathcal M}}

\newcommand{\cO}{{\mathcal O}}
\newcommand{\cP}{{\mathcal P}}

\newcommand{\cS}{{\mathcal S}}

\newcommand{\cU}{{\mathcal U}}
\newcommand{\cV}{{\mathcal V}}
\newcommand{\cW}{{\mathcal W}}
\newcommand{\cX}{{\mathcal X}}
\newcommand{\cY}{{\mathcal Y}}
\newcommand{\cZ}{{\mathcal Z}}

\newcommand{\CM}{\overline{\mathcal{M}}}

\newcommand{\ww}{\omega}

\newcommand{\delbar}{\overline{\partial}}

\newcommand{\delbarinv}{\smash{\overline{\partial}}\vphantom{\partial}^{-1}}

\newcommand{\hdelbar}{
	\mathchoice
	{
		\hat{\overline{\partial}}
	}
	{
		\scalerel*{\hat{\overline{\partial}}}{\hat{M}}
	}
	{}{}
}

\newcommand{\abs}[1]{\lvert#1\rvert}
\newcommand{\norm}[1]{\left\Vert#1\right\Vert}

\DeclareMathOperator{\codim}{codim}
\DeclareMathOperator{\rank}{rank}
\DeclareMathOperator{\dR}{dR}

\DeclareMathOperator{\GW}{GW}
\DeclareMathOperator{\PD}{PD}

\newcommand{\ssc}{\text{sc}}

\newcommand{\dmlog}{\smash{\overline{\mathcal{M}}}\vphantom{\mathcal{M}}^{\text{log}}}

\newcommand{\supp}{\operatorname{supp}}
\newcommand{\id}{\operatorname{id}}
\newcommand{\sign}{\text{sgn}}
\newcommand{\pr}{\operatorname{pr}}

\begin{document}


\title[The {S}teenrod problem and polyfold invariants]{The {S}teenrod problem for orbifolds \\ and polyfold invariants as intersection numbers}
\author{Wolfgang Schmaltz}
\date{\today}
\thanks{Research partially supported by Project C5 of SFB/TRR 191 ``Symplectic Structures in Geometry, Algebra and Dynamics,'' funded by the DFG}

\address{Faculty of Mathematics, Ruhr-Universit{\"a}t Bochum, 44801 Bochum, Germany}
\email{\href{mailto:wolfgang.schmaltz@rub.de}{wolfgang.schmaltz@rub.de}}
\urladdr{\url{https://sites.google.com/view/wolfgang-schmaltz/home}}


\begin{abstract}
	The Steenrod problem for closed orientable manifolds was solved completely by Thom.
	Following this approach, we solve the Steenrod problem for closed orientable orbifolds, proving that the rational homology groups of a closed orientable orbifold have a basis consisting of classes represented by suborbifolds whose normal bundles have fiberwise trivial isotropy action.
	
	Polyfold theory, as developed by Hofer, Wysocki, and Zehnder, has yielded a well-defined Gromov--Witten invariant via the regularization of moduli spaces.
	As an application, we demonstrate that the polyfold Gromov--Witten invariants, originally defined via branched integrals, may equivalently be defined as intersection numbers against a basis of representing suborbifolds.
\end{abstract}

\maketitle

\tableofcontents

\section{Introduction}

\subsection[The Steenrod problem]{The {S}teenrod problem}

The Steenrod problem was first presented in \cite{eilenberg1949problems} and asked the following question:
\textit{Can any homology class of a finite polyhedron be represented as an image of the fundamental class of some manifold?}
In \cite{thom1954quelques},\footnote{The reader should be advised that the commonly available English translation of this paper introduces a few errors which are \textit{not} present in the original version.} Thom conclusively answered this problem, completely solving it for closed orientable manifolds.
\begin{theorem}[{\cite[Thm.~II.1]{thom1954quelques}}]\label{thm:thom}
	The rational homology groups of a closed orientable manifold have a basis consisting of classes represented by closed embedded submanifolds.
\end{theorem}
A key insight into the problem was the multiplication theorem \cite[Thm.~II.25]{thom1954quelques} which demonstrated a cohomological relationship between two different types of classifying spaces---the Eilenberg--MacLane spaces and the universal bundles.
For solving this problem, and for his related work inventing cobordism theory, Thom was awarded the Fields medal in 1958.


In order to solve the Steenrod problem for orbifolds, for the most part we are able to follow the same approach as Thom.
Following this approach requires updating and reformulating analogs of a few classical results of the differential topological study of manifolds into the modern language of ep-groupoids.

\begin{theorem}[The {S}teenrod problem for closed orientable orbifolds]\label{thm:intro-steenrod-problem}
	The rational homology groups of a closed orientable orbifold have a basis consisting of classes represented by closed embedded full suborbifolds whose normal bundles have fiberwise trivial isotropy action.
\end{theorem}

Embedded full suborbifolds whose normal bundles have fiberwise trivial isotropy action are well-suited for general intersection theories.
Given such a suborbifold, the underlying topological space of the normal bundle is a vector bundle over the underlying topological space of the suborbifold.
In contrast, the underlying topological space of an arbitrary orbifold bundle will generally not be a vector bundle.
This means it is possible to use single valued sections (as opposed to multisections) for arguments involving perturbations.

We consider the Steenrod problem for orbifolds and give a self-contained introduction to some basic orbifold theory in \S~\ref{sec:steenrod-problem}.
In \S~\ref{subsec:orbifolds-ep-groupoids} we describe orbifolds using the modern language of ep-groupoids, prove a Whitney approximation theorem, and define the full embedded suborbifolds.
In \S~\ref{subsec:elementary-orbifold-alg-top} we compare cohomology theories on orbifolds, and discuss Poincar\'e duality.
In \S~\ref{subsec:elementary-orbifold-transversality} we provide some relevant transversality results.
In \S~\ref{subsec:solving-the-steenrod-problem} we solve the Steenrod problem for orbifolds.
The proof of Theorem~\ref{thm:intro-steenrod-problem} follows immediately from the more complete and technical Theorem~\ref{thm:steenrod-problem}.


\subsection{Polyfold invariants}

A foundational problem in symplectic geometry is obtaining well-defined invariants via the compactified moduli spaces that arise in the study of $J$-holomorphic curves.
To do this, we would like to think of a compactified moduli space as a space which possesses structure similar to that of a manifold, structure such as:  possession of a ``fundamental class,'' or the ability to integrate differential forms, or the ability to define notions of transversal intersection and intersection number.
However, for general symplectic manifolds, compactified moduli spaces will only have the structure of metrizable compact topological spaces---and this is insufficient for obtaining invariants.

Polyfold theory, developed by Hofer, Wysocki, and Zehnder, is a relatively new approach to solving this problem, and has been successful in proving the following regularization theorem.
\begin{theorem}[Polyfold regularization theorem, {\cite[Thm.~15.4, Cor.~15.1]{HWZbook}}]
	Consider a compact topological space $\CM$ which is equal to the zero set of a $\ssc$-smooth Fredholm section $\delbar$ of a strong polyfold bundle $\cW \to \cZ$, i.e., $\CM= \delbarinv (0) \subset \cZ$.
	
	There exist a class of ``abstract perturbations'' $p$ such that the perturbed zero set $(\delbar + p)^{-1}(0)$ has the structure of a compact oriented ``weighted branched orbifold.''
	
	Furthermore, different choices of abstract perturbations result in perturbed zero sets which are cobordant:
	given abstract perturbations $p_1$, $p_2$ there exists a compact oriented weighted branched orbifold $\cB$ with boundary:
		\[
		\partial \cB = -(\delbar+p_1)^{-1}(0) \sqcup (\delbar+p_2)^{-1}(0).
		\]
\end{theorem}

A {manifold} is locally homeomorphic to an open subset of $\R^n$, while an {orbifold} is locally homeomorphic to the quotient of an open subset of $\R^n$ by a finite group action.
Similarly, a {branched orbifold} is locally homeomorphic to the quotient of a finite union of open subsets of $\R^n$ by a finite group action.
The term ``branched orbifold'' refers to the smooth structures of the underlying topological space, while the ``weighted'' adjective refers to additional data used to define invariants such as the branched integral.
Crucially, weighted branched orbifolds possess enough structure to define the ``branched integration'' of differential forms.

\begin{theorem}[Polyfold invariants as branched integrals, {\cite[Cor.~15.2]{HWZbook}}]
	Consider a $\ssc$-smooth map 
		$f:\cZ \to \cO$
	from a polyfold $\cZ$ to an orbifold $\cO$.
	We may define the \textbf{polyfold invariant} as the homomorphism obtained by pulling back a de Rahm cohomology class from the orbifold and taking the branched integral over a perturbed zero set:
		\[
		H^*_{\dR} (\cO) \to \R,	\qquad \ww \mapsto \int_{\cS(p)} f^*\ww.
		\]
	By a version of Stokes' theorem~\ref{thm:stokes} for weighted branched orbifolds, this homomorphism does not depend on the choice of abstract perturbation used to obtain the weighted branched orbifold $\cS(p)$.
\end{theorem}


We prove that the polyfold invariants may equivalently be defined as intersection numbers.
Polyfold invariants take as input (co)homological data coming from a closed orientable orbifold.
The Steenrod problem for closed orientable orbifolds demonstrates how to take this homological data and realize it as the fundamental class of a full embedded suborbifold---which is suitable data for defining the polyfold invariants as intersection numbers.

Consider a weighted branched suborbifold $\cS$, an orbifold $\cO$, an embedded full suborbifold $\cX\subset \cO$ whose normal bundle has fiberwise trivial isotropy action, and a smooth map $f :\cS \to \cO$:
	\[
	\begin{tikzcd}[row sep=tiny]
	\cS \arrow[r,"f"] & \cO \arrow[d,phantom,"\cup"] \\
	& \cX.
	\end{tikzcd}
	\]
Weighted branched orbifolds possess oriented tangent spaces, and hence it is possible to formulate a notion of transversal intersection of the map $f$ with the suborbifold $\cX$.
Furthermore, transversal intersection is generic, and may be obtained via either of the following approaches:
\begin{itemize}
	\item through perturbation of the suborbifold $\cX$ (see Proposition~\ref{prop:transversality-perturbation-full-suborbifold}),
	\item through construction of an abstract perturbation (see Proposition~\ref{prop:transversality-regular-perturbation}).
\end{itemize}
(The second approach requires the additional hypothesis that the map $f$ extended to the ambient polyfold is a submersion.)

\begin{theorem}[Polyfold invariants as intersection numbers]\label{thm:intro-polyfold-invariants-intersection}
	Consider a compact oriented weighted branched suborbifold $\cS$, an oriented orbifold $\cO$, a closed embedded full suborbifold $\cX\subset \cO$ whose normal bundle has fiberwise trivial isotropy action, and a smooth map $f :\cS \to \cO$.
	
	Without loss of generality, we may assume that $f$ is transverse to $\cX$.
	There exists a well-defined \textbf{intersection number}:
		\[
		f|_{\cS} \cdot \cX.
		\]
	
	The Steenrod problem for orbifolds \ref{thm:steenrod-problem} guarantees the existence of a basis $\{[\cX_i]\}$ of $H_*(\cO;\Q)$ which consists of the fundamental classes of closed embedded full suborbifolds whose normal bundles have fiberwise trivial isotropy action.
	We may define the \textbf{polyfold invariant} as the homomorphism obtained by evaluating the intersection number on this basis of representing suborbifolds and linear extension:
		\[
		H_*(\cO;\Q) \to \Q, \qquad	\sum_i k_i[\cX_i] \mapsto \sum_i k_i \left( f|_{\cS} \cdot \cX_i \right).
		\]
\end{theorem}

When $\dim \cS + \dim \cX = \dim \cO$, the intersection number is given by the signed weighted count of a finite number of points of intersection (see Definition~\ref{def:intersection-number}).

Provided with a well-defined branched integral and a well-defined intersection number, a proof of the equality of the invariants involves little more than using Poincar\'e duality for orbifolds and a local comparison of these invariants at the finite points of intersection.
We use this equality to justify the assertion that the intersection number is an invariant, and does not depend on the choice of abstract perturbation, nor on the choice of basis of representing suborbifolds (see Remark~\ref{rmk:invariance-intersection-number}).

\begin{theorem}[Equivalence of the polyfold invariants]\label{thm:equality-polyfold-invariants}
	Let $\cS$ be a compact oriented weighted branched suborbifold. Let $\cO$ be a closed oriented orbifold and let $\cX \subset \cO$ be a closed oriented embedded full suborbifold whose normal bundle has fiberwise trivial isotropy action.
	Consider a smooth map $f: \cS \to \cO$ and assume without loss of generality that $f$ is transverse to $\cX$.
	
	The branched integral and the intersection number are related by the following equation:
	\[
	\int_{\cS} f^* \PD ([\cX]) = f|_{\cS} \cdot \cX.
	\]
\end{theorem}

We show how the Steenrod problem for orbifolds may be used to define the polyfold invariants in terms of intersection numbers in \S~\ref{sec:defining-polyfold-invariants}.
In \S~\ref{subsec:weighted-branched-suborbifolds} we discuss the structure of the weighted branched suborbifolds which arise as perturbed solution spaces in polyfold theory.
In \S~\ref{subsec:branched-integrals} we review details regarding the branched integral and the associated polyfold invariants.
In \S~\ref{subsec:intersection-numbers} we show that transversality is generic and define the polyfold invariants as intersection numbers.
In \S~\ref{subsec:equality-polyfold-invariants} we prove Theorem~\ref{thm:equality-polyfold-invariants}, showing that the polyfold invariants are equivalent and establish that the intersection number is an invariant.


\subsection[Application: The polyfold Gromov--Witten invariants]{Application: The polyfold {G}romov--{W}itten invariants}

The earliest interpretations of the Gromov--Witten invariants present in the literature were given in terms of counting a finite number of curves \cites{mcduffsalamon2012jholomorphic, ruan1994symplectic, ruan1996topological}.
For example, Ruan described the Gromov--Witten invariants as a finite sum, counted with multiplicity, of nonmultiple cover $J$-spheres in $\cM^*_{(A,J)}$ which intersect representatives of given cycles in the symplectic manifold \cite[Thm.~A]{ruan1996topological}.

However, such definitions have previously been restricted to genus zero Gromov--Witten invariants in semipositive symplectic manifolds.
In addition to technical issues regarding the regularization of the Gromov--Witten moduli spaces for general symplectic manifolds and arbitrary genus, we remark that in the genus $0$ case the Grothendieck--Knudsen spaces $\dmlog_{0,k}$ are finite-dimensional manifolds.
In contrast, if genus $g >0$ the general Deligne--Mumford spaces $\dmlog_{g,k}$ are orbifolds.
Therefore, in the genus $0$ case extant methods---such as representing a homology class as a pseudocycle in a manifold \cite{zinger2008pseudocycles}  or, indeed, the Steenrod problem for manifolds---were sufficient to interpret the Gromov--Witten invariants via intersection theory.

Polyfold theory has been successful in giving a well-defined Gromov--Witten invariant for $J$-holomorphic curves of arbitrary genus, and for all closed symplectic manifolds \cite{HWZGW}. As opposed to defining these invariants via intersection theory, the invariants are defined by pulling back de Rahm cohomology classes and branched integration. We now describe this precisely.

Let $(Q,\ww)$ be a closed symplectic manifold, and fix a homology class $A \in H_2 (Q;\Z)$ and integers $g, k\geq 0$ such that $2g+k\geq 3$.
Consider the following diagram of smooth maps between the perturbed Gromov--Witten moduli space $\cS_{A,g,k}(p)$, the $k$-fold product manifold $Q^k$, and the Deligne--Mumford orbifold $\dmlog_{g,k}$:  
	\[
	\begin{tikzcd}[column sep=12ex]
	\cS_{A,g,k}(p) \arrow{r}{ev_1\times\cdots\times ev_k} \arrow{d}{\pi} & Q^k\\
	\dmlog_{g,k}.
	\end{tikzcd}
	\]
Here $ev_i$ is evaluation at the $i$th-marked point, and $\pi$ is the projection map to the Deligne--Mumford space which forgets the stable map solution and stabilizes the resulting nodal Riemann surface by contracting unstable components.

Consider homology classes $\alpha_1,\ldots, \alpha_k \in H_* (Q;\Q)$ and $\beta\in H_* (\dmlog_{g,k};\Q)$.
We can represent the Poincar\'e duals of the $\alpha_i$ and $\beta$ by closed differential forms in the de Rahm cohomology groups, $\PD(\alpha_i)\in H^*_{\dR} (Q)$ and $\PD(\beta)\in H^*_{\dR}(\dmlog_{g,k})$.
By pulling back via the evaluation and projection maps, we obtain a closed $\ssc$-smooth differential form 
	\[
	ev_1^* \PD (\alpha_1) \wedge \cdots \wedge ev_k^* \PD(\alpha_k) \wedge\pi^* \PD (\beta) \in H^*_{\dR} (\cZ_{A,g,k}).
	\]
The \textbf{polyfold Gromov--Witten invariant} is the homomorphism
	\[
	\GW^Q_{A,g,k} : H_* (Q;\Q)^{\otimes k} \otimes H_* (\dmlog_{g,k}; \Q) \to \R
	\]
defined in \cite[Thm.~1.12]{HWZGW} via the branched integral
	\[
	\GW^Q_{A,g,k} (\alpha_1,\ldots,\alpha_k;\beta) : = \int_{\cS_{A,g,k}(p)} ev_1^* \PD (\alpha_1) \wedge \cdots \wedge ev_k^* \PD(\alpha_k) \wedge\pi^* \PD (\beta).
	\] 

Via Theorems~\ref{thm:intro-polyfold-invariants-intersection} and \ref{thm:equality-polyfold-invariants}, we immediately obtain the following equivalent description of the polyfold Gromov--Witten invariants as an intersection number.

\begin{corollary}[Gromov--Witten invariants as intersection numbers]
	The polyfold Gromov--Witten invariant may equivalently be defined as the intersection number evaluated on a basis of representing submanifolds $\cX \subset Q$ and representing suborbifolds $\cB \subset \cO$:
		\[
		\GW^Q_{A,g,k} ([\cX_1],\ldots,[\cX_k];[\cB]) :=
		\left(ev_1\times\cdots\times ev_k\times\pi\right)|_{\cS_{A,g,k}(p)} \cdot \left(\cX_1 \times\cdots\times \cX_k \times \cB\right).
		\]
	The invariant does not depend on the choice of abstract perturbation, nor on the choice of representing basis.		
\end{corollary}

This definition aligns with the traditional geometric interpretation of the Gromov--Witten
invariants as a count of curves which at the $i$th-marked point passes through $\cX_i$ and such
that the image under the projection $\pi$ lies in $\cB$:
	\[
	\begin{tikzcd}[column sep= huge, row sep=tiny]
	\cS_{A,g,k}(p) \arrow[r,"ev_1\times \cdots \times ev_k \times \pi"] & Q^k \times \dmlog_{g,k} \arrow[d,phantom,"\cup\quad"] \\
	& \cX_1\times \cdots \times \cX_k \times \cB.
	\end{tikzcd}
	\]

\begin{remark}
	As we will show in subsequent work, the projection to the Deligne-- Mum\-ford orbifold,
		\[
		\pi: \cZ_{A,g,k} \to \dmlog_{g,k},
		\]
	is not a submersion. It follows that the map 
		\[
		ev_1\times \cdots \times ev_k \times \pi : \cZ_{A,g,k} \to Q\times \cdots \times Q \times \dmlog_{g,k}
		\]
	does not satisfy the hypotheses of Proposition~\ref{prop:transversality-regular-perturbation}.  Hence transversality of the map $ev_1\times \cdots \times ev_k \times \pi$ with a representing suborbifold $\cX_1\times \cdots \times \cX_k \times \cB$ may only be obtained via perturbation of the suborbifold as in Proposition~\ref{prop:transversality-perturbation-full-suborbifold}.
\end{remark}


\section[The Steenrod problem for orbifolds]{The {S}teenrod problem for orbifolds}\label{sec:steenrod-problem}

We begin this section with a self-contained introduction to some of the basic theory for orbifolds.
Along the way, we take advantage of the opportunity to collect a handful of elementary observations and results regarding orbifolds, which are either undocumented or scattered throughout the literature.
This ultimately culminates in a proof of the Steenrod problem for orbifolds.

\subsection{Orbifolds and ep-groupoids}\label{subsec:orbifolds-ep-groupoids}

The notion of orbifold was first introduced by Satake \cites{satake1956generalization}, with further descriptions in terms of groupoids and categories by Haefliger \cites{haefliger1971homotopy, haefliger1984groupoide, haefliger2001groupoids}, and Moerdijk \cites{moerdijk2002orbifolds, moerdijk2003introduction}.

Recently, Hofer, Wysocki, and Zehnder have utilized the language of ep-groupoids in the development polyfold theory.
In the present context we do not need the full strength of this theory; however the polyfold literature (in particular, \cite{HWZbook}) provides a large body of basic definitions and results on orbifolds and ep-groupoids.

\begin{definition}[{\cite[Defs.~7.1,~7.3]{HWZbook}}]
	A \textbf{groupoid} $(O,\bm{O})$ is a small category consisting of a set of objects $O$, a set of morphisms $\bm{Z}$ which are all invertible, and the five structure maps $(s,t,m,u,i)$ (the source, target, multiplication, unit, and inverse maps).
	An  \textbf{ep-groupoid} is a groupoid $(O,\bm{O})$ whose object set $O$ and morphism set $\bm{O}$ both have the structure of finite-dimensional manifolds,
	and such that all the structure maps are smooth maps and which satisfy the following properties.
	\begin{itemize}
		\item \textbf{(\'etale).}  The source and target maps
		$s:\bm{O}\to O$ and $t:\bm{O}\to O$ are surjective local diffeomorphisms.
		\item \textbf{(proper).}  For every point $x\in O$, there exists an
		open neighborhood $V$ of $x$ such that the map
		$t:s^{-1}(\overline{V})\rightarrow O$ is a proper mapping.
	\end{itemize}
\end{definition}

For a fixed object $x\in O$ we denote the \textbf{isotropy group of $x$} by
	\[
	\bm{G}(x) := \{	\phi \in \bm{O} \mid s(\phi)=t(\phi = x)	\}.
	\]
The properness condition ensures that this is a finite group.
The \textbf{non-effective part} of the isotropy group $\bm{G}(x)$ is the subgroup which acts trivially on a local uniformizer, i.e.,
	\[
	\bm{G}^\text{non-eff}(x) := \{	\phi \in \bm{G}(x)	\mid	\phi (U) = U	\},
	\]
while the \textbf{effective isotropy group} is the quotient group
	\[
	\bm{G}^\text{eff}(x) := \bm{G}(x) / \bm{G}^\text{non-eff}(x).
	\]
The \textbf{orbit space} of an ep-groupoid $(O,\bm{O})$,
	\[
	\abs{O} := O / \sim,
	\]
is the quotient of the set of objects $O$ by the equivalence relation $x\sim x'$ if there exists a morphism $\phi\in \bm{O}$ such that $s(\phi)=x$ and $t(\phi)=x'$.  It is equipped with the quotient topology defined via the map $\pi: O\to\abs{O}, x\mapsto \abs{x}.$

\begin{definition}
	Let $\cO$ be a second-countable, paracompact, Hausdorff topological space.  An \textbf{orbifold structure} on $\cO$ consists of an ep-groupoid $(O,\bm{O})$ and a homeomorphism $\abs{O}\simeq \cO$ (compare with \cite[Def.~16.1]{HWZbook}).
\end{definition}

An \textbf{orientation} of an orbifold structure consists of a choice of orientation on the object space and on the morphism space, such that the source and target maps are orientation preserving local diffeomorphisms.

Defining an ep-groupoid involves making a choice of local structures.  Taking an equivalence class of ep-groupoids makes our differentiable structure choice independent.  The appropriate notion of equivalence in this category-theoretic context is a ``Morita equivalence class'' (see \cite[Def.~16.2]{HWZbook}).

\begin{definition}
	An \textbf{orbifold} consists of a second-countable, paracompact, Hausdorff topological space $\cO$ together with a Morita equivalence class of orbifold structures $[(O,\bm{O})]$ on $\cO$ (compare with \cite[Def.~16.3]{HWZbook}).
\end{definition}

Taking a Morita equivalence class of a given orbifold structure is analogous to taking a maximal atlas for a given atlas in the usual definition of a manifold.

We say that that $\cO$ is \textbf{closed} if the underlying topological space is compact, and if the object and morphism spaces are boundaryless.
The orbifold $\cO$ is \textbf{orientable} if a representative orbifold structure $(O,\bm{O})$ can be given an orientation.

\begin{definition}\label{def:manifold-as-trivial-orbifold}
	Consider a manifold $\cM$. We may define a \textbf{trivial orbifold structure} on $\cM$ as follows:
	\begin{itemize}
		\item define the object set by $M:=\cM$,
		\item define the morphism set by $\bm{M} := \{ \id_x : x\to x \mid x\in \cM\}$.
	\end{itemize}
	It is obvious this defines an orbifold structure on the underlying topological space $\cM$.
\end{definition}

\begin{notation}
	It is common to denote both an ep-groupoid $(O,\bm{O})$, and its object set $O$, by the same letter ``$O$.''	
	We will refer to the underlying set, the underlying topological space, or the orbifold by the letter ``$\cO$.''
	Furthermore, we will write objects as ``$x\in O$,'' morphisms as ``$\phi \in \bm{O}$,'' and points as ``$[x]\in \cO$'' (due to the identification $\abs{O} \simeq \cO$).
\end{notation}

The following proposition shows the relationship between the more classically familiar definition of an orbifold (defined in terms of the local topology), and our abstract formulation of an orbifold in terms of ep-groupoids.

\begin{proposition}[Natural representation of $\bm{G}(x)$, {\cite[Thm.~7.1]{HWZbook}}] \label{prop:natural-representation}
	Consider an ep-groupoid $(O,\bm{O})$.  Let $x\in O$ with isotropy group $\bm{G}(x)$.  Then for every open neighborhood $V$ of $x$ there exists an open neighborhood $U\subset V$ of $x$, a group homomorphism $\Phi : \bm{G}(x)\rightarrow \text{Diff}(U)$, $g\mapsto  \Phi (g)$,  and a smooth map
	$\Gamma: \bm{G}(x)\times U\rightarrow \bm{O}$ such that the following holds.
	\begin{enumerate}
		\item $\Gamma(g,x)=g$.
		\item $s(\Gamma(g,y))=y$ and $t(\Gamma(g,y))=\Phi (g)(y)$ for all $y\in U$ and $g\in \bm{G}(x)$.
		\item If $h: y\rightarrow z$ is a morphism between points in $U$, then there exists a unique element $g\in \bm{G}(x)$ satisfying $\Gamma(g,y)=h$, i.e., 
		\[
		\Gamma: \bm{G}(x)\times U\rightarrow \{\phi\in \bm{O} \mid   \text{$s(\phi)$ and $t(\phi)\in U$}\}
		\]
		is a bijection.
	\end{enumerate}
	The data $(\Phi,\Gamma)$ is called the \textbf{natural representation} of $\bm{G}(x)$.
	The open neighborhood $U$ is called a \textbf{local uniformizer} centered at $x$.
\end{proposition}

Having given an abstract definition of an orbifold we now define maps between them.

\begin{definition}
	A \textbf{$C^k$ functor} between two ep-groupoids
	\[
	\hat{f}:(O,\bm{O}) \to (P,\bm{P})
	\]
	is a functor on groupoidal categories which moreover is an $C^k$ map when considered on the object and morphism sets.
\end{definition}

An $C^k$ functor between two orbifold structures $(O,\bm{O})$, $(P,\bm{P})$ with underlying topological spaces $\cO$, $\cP$ induces a continuous map on the orbit spaces $\abs{\hat{f}}:\abs{O} \to \abs{P}$, and hence also induces a continuous map $f : \cO \to \cP$, as illustrated in the following commutative diagram.
	\[
	\begin{tikzcd}[row sep=small]
	\abs{O} \arrow[d,dash,"\begin{sideways}$\sim$\end{sideways}"] \arrow[r,"\abs{\hat{f}}"] & \abs{P} \arrow[d,dash,"\begin{sideways}$\sim$\end{sideways}"] \\
	\cO \arrow[r,"f"] & \cP
	\end{tikzcd}
	\]

\begin{definition}
	Consider two topological spaces  $\cO$, $\cP$ with orbifold structures $(O,\bm{O})$, $(P,\bm{P})$. We define a \textbf{$C^k$ map between orbifolds} as a continuous map 
	\[
	f: \cO \to \cP
	\]
	between the underlying topological spaces of the orbifolds, for which there exists an associated $C^k$ functor
	\[
	\hat{f}: (O,\bm{O}) \to (P,\bm{P}).
	\]
	such that $\abs{\hat{f}}$ induces $f$.
\end{definition}

\begin{remark}
	From an abstract point of view a stronger notion of map is needed.   This leads to the definition of \textit{generalized maps} between orbifold structures, following a category-theoretic localization procedure \cite[\S~2.3]{HWZ3}.  Following this, a precise notion of map between two orbifolds is defined using an appropriate equivalence class of a given generalized map between two given orbifold structures \cite[Def.~16.5]{HWZbook}.	
	With this in mind, taking an appropriate equivalence class of a given $C^k$ functor between two given orbifold structures is sufficient for defining a map between two orbifolds.
\end{remark}

\begin{theorem}[Whitney approximation] \label{thm:whitney-approximation}
	Let $\cO$ be a topological space with orbifold structure $(O,\bm{O})$, and let $\cM$ be a manifold with trivial orbifold structure $(M,\bm{M})$. Consider a continuous map between the underlying topological spaces,
		$
		f: \cO \to \cM.
		$
	Then there exists a smooth functor
		$
		\hat{h} : (O,\bm{O}) \to (M,\bm{M})
		$
	such that the induced continuous map $h=\abs{\hat{h}} : \cO \to \cM$ is homotopic to $f$.
	Furthermore, if $f$ is proper, we may assume that $h$ is also proper.
\end{theorem}
\begin{proof}
We break the proof into four steps.

\begin{itemize}[leftmargin=0em]
\item[]\textbf{Step 1:} \emph{Associated to a continuous map between the the underlying topological spaces
			$f: \cO \to \cM$
			there exists a continuous functor $\hat{f}: (O,\bm{O}) \to (M,\bm{M})$
			such that $f = \abs{\hat{f}}$.}
\end{itemize}
Construct the continuous functor 
	$
	\hat{f}: (O,\bm{O}) \to (M,\bm{M})
	$
as follows:
\begin{itemize}
	\item on objects, define $O \to M$ by the composition $O \to \abs{O} \xrightarrow{f} \abs{M}\simeq M$,
	\item on morphisms, define $\bm{O} \to \bm{M}$ by $\bm{O}\xrightarrow{s} O \to \abs{O} \xrightarrow{f} \abs{M}\simeq\bm{M}$.
\end{itemize}
In both cases, the definition of the trivial orbifold structure $(M,\bm{M})$ gives the identifications $M\simeq\abs{M}\simeq \bm{M}$. 
From this definition it is clear that the induced continuous map $\abs{\hat{f}}$ is the same as the original map $f$.
	
Following step 1, the proof is (nearly) identical to the modern proof for smooth manifolds (see \cite[Thms.~6.21,~6.26]{lee2012introduction}). The main new ingredient is the existence of morphism invariant smooth partitions of unity (see \cite[Def.~7.16, Thm.~7.4]{HWZbook} for this existence).
	
\begin{itemize}[leftmargin=0em]
\item[]\textbf{Step 2:} \emph{Let $\delta: \cO \to \R^+$ be a positive continuous function. 
			Let $\hat{f}: (O,\bm{O}) \to \R^N$ be a continuous functor (where $\R^N$ has the trivial orbifold structure). Then there exists a smooth functor $\hat{g}:(O,\bm{O}) \to \R^N$ such that
			\[
			\norm{ f([x]) - g([x]) } < \delta([x]), \qquad \text{for all } [x]\in \cO.
			\]}
\end{itemize}
For any point $[x_0] \in \cO$, choose a local uniformizer $U(x_0)$ centered at a representative $x_0 \in O$.
By continuity of the composition $\delta\circ \pi : O \to \R^+$, and continuity of $\hat{f}$, there exists an open neighborhood $V(x_0)$ of $x_0$ such that
	\[
	\norm{\hat{f}(x) - \hat{f}(x_0)} < \delta([x]), \qquad \text{for all } x\in V(x_0).
	\]
By the $\bm{G}(x_0)$-invariance of $\hat{f}$, we may moreover assume the neighborhood $V(x_0)$ is $\bm{G}(x_0)$-invariant. Since the projection $\pi: O \to \cO, x\mapsto [x]$ is an open map, $\abs{V(x_0)}$ is an open neighborhood of $[x_0]\in \cO$.
Repeating this for any point $[x_i] \in \cO$ we obtain an open cover $\{\abs{V_i}\}_{i\in I}$ of $\cO$.
Denote by $V^*_i := \pi^{-1} (\pi(V_i))$ the saturation of each open set $V_i$; the collection $\{V^*_i\}_{i\in I}$ is a saturated open cover of the object space $O$.
Moreover, observe via the morphism invariance of $\hat{f}$ that
	\begin{equation}\label{eq:norm}
	\norm{\hat{f}(y)- f([x_i])}=\norm{\hat{f}(y) - \hat{f}(x_i)} < \delta([y]), \qquad \text{for all } y\in V^*_i.
	\end{equation}
		
Let $\beta_i : O \to [0,1], i\in I$ be a morphism invariant smooth partition of unity subordinate to the cover 
$\{V^*_i\}_{i\in I}$. Define 
	\begin{align*}
	\hat{g}: O &\to \R^k\\
			y & \mapsto \sum_{i\in I} \beta_i (y) f([x_i]).
	\end{align*}
Then $\hat{g}$ is a smooth function. Moreover, $\hat{g}$ is morphism invariant since each $\beta_i$ is morphism invariant. By sending all morphisms $\phi \in \bm{O}$ to the appropriate identity morphism (that is, $\hat{g}(\phi) := \id_{\hat{g} (s(\phi))}$) we see that $\hat{g}$ defines a smooth functor.
	
For any object $y\in O$, since $\sum_{i\in I} \beta_i \equiv 1$ we have
	\begin{align*}
	\norm{ \hat{f}(y) - \hat{g}(y)	}
	& = \norm{\left(	\sum_{i\in I} \beta_i(y)	\right) \hat{f}(y) - \sum_{i\in I} \beta_i (y) f([x_i])} \\
	& = \sum_{i\in I} \beta_i(y) \norm{	\hat{f}(y) - f([x_i])	} \\
	& \leq \delta([y]),
	\end{align*}
where the final inequality comes from equation \eqref{eq:norm} and the fact that $y \in \supp \beta_i$ only if $y \in V^*_i$.
	
\begin{itemize}[leftmargin=0em]
	\item[]\textbf{Step 3:} \emph{We define the smooth functor approximating $f$, following \cite[Thm.~6.26]{lee2012introduction} closely.}
\end{itemize}
By the Whitney embedding theorem \cite[Thm.~6.15]{lee2012introduction}, we may assume $\cM$ is a properly embedded submanifold of $\R^N$.
By \cite[Prop.~6.25]{lee2012introduction}, there exists a tubular neighborhood $U$ of $\cM$ in $\R^N$ and a smooth retraction $r: U \to \cM$.
	
Define $\delta' : \cM \to \R^+$ by
	\[
	\delta'(x) := \sup \{\ep\leq 1 \mid B_\ep (x)\subset U\},
	\]
as in \cite[Thm.~6.26]{lee2012introduction}, this function is continuous. Define a continuous function $\delta: \cO \to \R^+$ by the composition $\delta := \delta' \circ f$.
	
By step 2, there exists a smooth functor $\hat{g} : (O,\bm{O}) \to \R^N$ such that
	\[
	\norm{f([x])-g([x])} \leq \delta([x]) \qquad \text{for all } [x] \in \cO.
	\]
Let $H: \cO \times [0,1] \to \cM$ be the composition of $r$ with the straight line homotopy between $f([x])$ and $g([x])$:
	\[
	H([x],t) : = r\left(	(1-t) f([x]) + t g([x])	\right)
	\]
Note that $	(1-t) f([x]) + t g([x]) \in U$; by construction, $\norm{f([x])-g([x])} \leq \delta([x])= \delta'(f([x]))$ hence $g([x]) \in B_{\delta'([x])} (f([x])) \subset U$, and then use convexity of the ball.
	
Thus, $H$ is a homotopy between the maps $H(\cdot,0) = f$ and the map $H(\cdot,1) = r\circ g$.
Moreover, the composition $h:= r\circ g : \cO \to \cM$ has an associated smooth functor given by the composition
	\[
	\hat{h}: (O,\bm{O}) \xrightarrow{\hat{g}} (U,\bm{U}) \xrightarrow{\hat{r}} (M,\bm{M}).
	\]
where $\hat{r}: (U,\bm{U}) \to (M,\bm{M})$ is the trivial functor associated to the trivial orbifold structures on $U$ and $\cM$.
	
\begin{itemize}[leftmargin=0em]
	\item[]\textbf{Step 4:} \emph{We discuss the statement about properness.}
\end{itemize}
In this case, the reasoning is identical to the analogous statement for smooth manifolds. 
This completes the proof of the theorem.
\end{proof}

We now explain the appropriate definition of an orbifold bundle.
Let $(O,\bm{O})$ be an ep-groupoid, and consider a vector bundle over the object space, $P:E\to O.$
The source map $s:\bm{O}\to O$ is a local diffeomorphism, and hence we may consider the fiber product
	\[
	\bm{O} _s\times_P E = \{(\phi,e)\in \bm{O}\times E	\mid	s(\phi)=P(e)	\}.
	\]
We may also view as $\bm{O} _s\times_P E$ as the pullback bundle via $s$ over the morphism space $\bm{O}$,
	\begin{equation*}
		\begin{tikzcd}
		\bm{O} _s\times_P E \arrow[r] \arrow[d] & E \arrow[d] \\
		\bm{O} \arrow[r, "s"] & O.
		\end{tikzcd}
	\end{equation*}

\begin{definition}[{\cite[Def.~8.4]{HWZbook}}]\label{def:orbifold-bundle}
	A \textbf{bundle over an ep-groupoid} consists of a vector bundle over the object space $P:E\to O$ together with a bundle map
		\[
		\mu : \bm{O} _s\times_P E \to E
		\]
	which covers the target map $t:\bm{O} \to O$, such that the diagram
	\begin{equation*}
		\begin{tikzcd}
		\bm{O} _s\times_P E \arrow[r, "\mu"] \arrow[d] & E \arrow[d] \\
		\bm{O} \arrow[r, "t"] & O
		\end{tikzcd}
	\end{equation*}
	commutes.
	Furthermore we require the following:
	\begin{enumerate}
		\item $\mu$ is a surjective local diffeomorphism and linear on fibers,
		\item $\mu(\id_x,e)= e$ for all $x\in O$ and $e\in E_x$,
		\item $\mu(\phi \circ \g ,e)= \mu (\phi,\mu(\g,e))$ for all $\phi,\g\in\bm{O}$ and $e\in E$ which satisfy
				\[
				s(\g) = P(e),\qquad t(\g) = s(\phi) = P(\mu(\g,e)).
				\]
	\end{enumerate}
\end{definition}

Given a bundle over an ep-groupoid we may obtain an ep-groupoid $(E,\bm{E})$ as follows: take the original vector bundle $E$ as object set, and take $\bm{E} := \bm{O} _s\times_P E$ as morphism set.  Moreover, we have source and target maps $s,t :\bm{E} \to E$ defined as follows:
	\[
	s(\phi,e) := e, \qquad t(\phi,e) := \mu(\phi,e).
	\]
There is a natural smooth projection functor $\hat{P}: (E,\bm{E}) \to (O,\bm{O})$.

Given an ep-groupoid $(O,\bm{O})$ we can define \textbf{tangent bundle over an ep-groupoid} by taking the tangent bundle over the object space $P:TO\to O$ and by defining
	\[
	\mu : 	\bm{O} _s\times_P TO \to TO, \qquad	(\phi, e) \mapsto T\phi (e).
	\]

We now introduce what will become the central objects resulting from the proof of the Steenrod problem for orbifolds.

\begin{definition}\label{def:embedded-full-suborbifold}
	Let $(O,\bm{O})$ be an orbifold structure of dimension $n$.
	An \textbf{embedded full suborbifold structure} of dimension $k$ consists of a subgroupoid $(X,\bm{X}) \subset (O,\bm{O})$, whose objects and morphisms are given the subspace topology, and which satisfies the following conditions:
	\begin{enumerate}
		\item \label{def:suborbifold-condition-1} The subcategory $(X,\bm{X})$ is full and the object set is saturated, i.e.,
		\begin{itemize}
			\item for all objects $x,y\in X$, $\text{mor}_X (x,y) = \text{mor}_O (x,y)$,
			\item $X=\pi^{-1}(\pi(X))$, where $\pi: O \to \abs{O}, x\mapsto [x]$.
		\end{itemize}
		\item \label{def:suborbifold-condition-2} At any object $x\in O$ there exists a $\bm{G}(x)$-invariant neighborhood $U\subset O$ and a chart $\phi: U\to \R^n$; it follows that $X\cap U$ is also $\bm{G}(x)$-invariant. We then require that $\phi(X\cap U)$ is the intersection of a $k$-dimensional plane of $\R^n$ with $\phi(U)$.
	\end{enumerate}
	We then say that $\cX\simeq \abs{X}$ with the Morita equivalence class $[(X,\bm{X})]$ is an \textbf{embedded full suborbifold} of the orbifold $\cO\sim \abs{O}$ with Morita equivalence class $[(O,\bm{O})]$.
\end{definition}

With this definition $(X,\bm{X})$ is naturally a $k$-dimensional orbifold structure in its own right.
The object set $X$ has the structure of a $k$-dimensional submanifold of the object set $O$, considered with respect to the manifold structure on $O$. Charts are given by the restrictions $\phi|_{X\cap U} : X\cap U \to \R^k$ and smoothness of the transition maps is then inherited from the smoothness of the transition maps for $O$. In a similar way, one can give the morphism set $\bm{X}$ the structure of a $k$-dimensional submanifold, and check that the \'etale property for $(X,\bm{X})$ is induced from the \'etale property of $(O,\bm{O})$.
One can see properness using the following fact: if a map $f:X\to Y$ is proper, then for any subset $V\subset Y$ the restriction $f|_{f^{-1}(V)}: f^{-1}(V)\to V$ is proper.

With these orbifold structures, the inclusion functor 
	$\hat{i}: (X,\bm{X}) \hookrightarrow (O,\bm{O})$
is a smooth embedding on both objects and morphisms.

Since the subcategory $(X,\bm{X})$ is full, observe that for an object $x\in X$ the isotropy groups, considered with respect to $X$ and with respect to $O$, are identical, i.e.,
	\[
	\{ \phi \in \bm{X} \mid s(\phi)=t(\phi) = x	\} = \{	\phi \in \bm{O} \mid s(\phi)=t(\phi) = x\}.
	\]
We may therefore denote both isotropy groups as $\bm{G}(x)$ without ambiguity.

\begin{remark}
	For a somewhat more general definition of suborbifold we refer to \cite[Def.~2.4]{adem2007orbifolds}, see also \cite[Def.~3.1]{cho2013orbifold}.
	In contrast to the above definition of a full suborbifold, the isotropy groups for a general suborbifold $\cX$ of an orbifold $\cO$ will not be identical. Instead, at an object $x\in X\subset O$ there exists a $\bm{G}_\cO(x)$-invariant neighborhood $U\subset O$ and a chart $\phi: U\to \R^n$, such that $\phi(X\cap U)$ consists of a {union} of precisely $\sharp \bm{G}^\text{eff}_\cO (x) / \sharp \bm{G}^\text{eff}_\cX(x)$-many
	$k$-dimensional planes intersected with $\phi(U)$.
\end{remark}

There exist definitions of Riemannian metrics and normal bundles for orbifolds; in particular, we may consider the normal bundle of an embedded full suborbifold (see \cite[Appx.]{chen2001orbifoldgromovwitten}).

\begin{definition}
	Let $\cX$ be an embedded full suborbifold of an orbifold $\cO$, and consider the normal bundle $N\cX$.
	We say that $N\cX$ has \textbf{fiberwise trivial isotropy action} if for every object $x\in X$ the isotropy group $\bm{G}(x)$ acts trivially on the normal space $N_x X$ at $x$.
\end{definition}

In terms of Definition~\ref{def:orbifold-bundle}, $N\cX$ has fiberwise trivial isotropy action if $\mu(\phi,e) = e$ for any isotropy morphism $\phi\in\bm{G}(x)$.

We can formulate a tubular neighborhood theorem as follows.
For a sufficiently small open neighborhood $N_\ep\cX$ of the zero section $\cX \to N\cX$ there exists a well-defined map between orbifolds 
		\[i: N_\ep\cX \to \cO\]
which is a homeomorphism onto its image when considered with respect to the underlying topological spaces, and which is a local diffeomorphism when considered with respect to the orbifold structures.

\begin{remark}\label{rmk:orbit-space-normal-bundle}
	It follows from the definition that the underlying topological space $N\cX$ of a normal bundle with fiberwise trivial isotropy action is in fact a vector bundle over the underlying topological space $\cX$.  
	This is in contrast to the orbit space of an arbitrary orbifold bundle---in general, the quotient of a fiber by a nontrivial isotropy action will no longer be a vector space.
\end{remark}

\subsection{Elementary orbifold algebraic topology}\label{subsec:elementary-orbifold-alg-top}

From the beginning, Satake gave definitions for the natural analog of de Rahm cohomology for orbifolds, and recognized that closed oriented orbifolds satisfy Poincar\'e duality with respect to this cohomology \cite{satake1956generalization}.
Given an orbifold $\cO$ with an orbifold structure $(O,\bm{O})$ the \textbf{de Rahm complex} is be defined as the space of differential forms on the object space $O$ which are morphism invariant:
	\[
	\Omega_{ep}^* (O) := \{	\ww\in \Omega^*(O) \mid T\phi^* \ww_y = \ww_x \text{ for every morphism } \phi:x\to y	\}.
	\]
The exterior derivative preserves morphism invariance, and hence we can consider the \textbf{de Rahm cohomology}
	$
	H^*_{\dR} (\cO) := H^*(\Omega_{ep}^* (O), d)
	$
of an orbifold.

\begin{theorem}[Comparison of orbifold cohomology theories]
	Consider an orbifold $\cO$.
	In addition to the de Rahm cohomology, we may consider the \v{C}ech cohomology (with real coefficients) of the underlying topological space, denoted $\check{H}^*(\cO; \R)$; and the singular cohomology (with real coefficients) of the underlying topological space, denoted $H^*(\cO; \R)$. These cohomology theories are naturally isomorphic:
		\[
		H^*_{\dR} (\cO) \simeq \check{H}^*(\cO; \R) \simeq H^*(\cO; \R).
		\]
\end{theorem}
\begin{proof}
	The first isomorphism $H^*_{\dR} (\cO) \simeq \check{H}^*(\cO; \R)$ was proved by Satake \cite[Thm.~1]{satake1956generalization}. The second isomorphism $\check{H}^*(\cO; \R) \simeq H^*(\cO; \R)$ follows from the fact that orbifolds admit a triangulation (see \cite[Prop.~1.2.1]{moerdijk1999simplicial}) and from the fact that \v{C}ech and singular homology coincide for triangulable spaces (see \cite[Ch.~IX Thm.~9.3]{eilenberg1952foundations}).
\end{proof}

\begin{theorem}[{\cite[Thm.~3]{satake1956generalization}}]
	Let $\cO$ be a closed oriented orbifold of dimension $n$. 
	With respect to the real \v{C}ech (co)homology of $\cO$ (and hence also real singular (co)homology) there exists a natural Poincar\'e duality isomorphism
		\[
		\check{H}^*(\cO;\R) \simeq \check{H}_{n-*}(\cO;\R).
		\]
\end{theorem}

Via the isomorphisms $H^*(\cO;\Q)\otimes \R \simeq H^*(\cO;\R)$, $H_*(\cO;\Q)\otimes \R \simeq H_*(\cO;\R)$, observe compact orbifolds satisfy Poincar\'e duality for \v{C}ech and singular cohomology with rational coefficients.
Orbifolds can moreover be considered as \textit{rational homology manifolds}, in the sense that for any $[x]\in\cO$,
	\[
	H_* (\cO, \cO \setminus \{[x]\}; \Q) = 
	\begin{cases}
	\Q 	& *=0,\dim \cO \\
	0	& \text{else}.
	\end{cases}
	\]

\begin{remark}[Failure of Poincar\'e duality for general coefficients]
	In spite of the above fact that compact orbifolds satisfy Poincar\'e duality for rational coefficients, for general coefficients Poincar\'e duality fails. In a recent paper by Lange, a classification theorem establishes that the underlying space of an orbifold is a topological manifold if and only if all local isotropy groups are of a specific form \cite[Thm.~A]{lange2015underlying}.
	Hence, any global quotient orbifold of the form $S^n / G$ for a group $G$ which is not of this specific form will fail to satisfy Poincar\'e duality for general choice of coefficients.
\end{remark}

\begin{remark}[Alternative orbifold cohomology theories]
	The cohomology theories we have considered thus far depend only on the topology of the underlying topological space of an orbifold.
	In \cite{chen2004newcohomology} an exotic cohomology theory (called \textit{Chen--Ruan cohomology}) of orbifolds is defined by formulating an appropriate generalization of quantum cohomology to orbifolds, and then restricting to the degree-zero portion of this cohomology to recover a new cohomology theory.
	In contrast to the above cohomology theories, this cohomology theory incorporates additional data about the local isotropy groups.
\end{remark}

\subsubsection[Relationship between the Thom class and the Poincar\'e dual of a suborbifold]{Relationship between the {T}hom class and the {P}oincar\'e dual of a suborbifold}

Recall the defining properties of the Thom class in topological terms. Consider an oriented rank $k$ vector bundle $\pi: E\to B$ over a paracompact topological space.	Then there exists a unique cohomology class with compact vertical support $u_B \in H^k_{cv} (E;\Z)$ called the \textbf{Thom class}, and which is characterized by the following properties.
\begin{itemize}
	\item For every $x\in B$ the restriction of $u_B$ to $H^k_{cv}(E_x;\Z)\simeq\Z$ is the generator determined by the orientation of $E$.
	\item The map
		\begin{align*}
			H^*(B;\Z) &\to H^{*+k}_{cv}(E;\Z) \\
			\alpha & \mapsto \pi^*\alpha \cup u_B
		\end{align*}
		is an isomorphism.
\end{itemize}

In the current situation of orbifolds, we can rephrase these properties in terms of de Rahm classes.  This can be used to give an explicit description of the Poincar\'e dual of an embedded full suborbifold.

Let $\cX$ be a closed embedded full suborbifold of an orbifold $\cO$.
Consider the normal bundle $N\cX$ of rank $k:=\dim \cO - \dim \cX$, and suppose that it is oriented and has fiberwise trivial isotropy action.	
In terms of de Rahm cohomology, the Thom class is the unique cohomology class $\tau\in H^k_{\dR,cv} (N\cX)$, i.e., a de Rahm form with compact vertical support, 
and which is characterized by the following properties.
\begin{itemize}
	\item For every $[x]\in \cX$ the restriction of $\tau$ to $H_{\dR,c}^k (N_{[x]} \cX)$ is the generator determined by the orientation of $N\cX$. Here $H_{\dR,c}^*$ denotes the space of de Rahm forms with compact support.
	\item The map
		\begin{align*}
		H_{\dR}^* (\cX) &\to H_{\dR,cv}^{*+k} (N\cX) \\
		\ww &\mapsto \pi^*\ww \wedge \tau
		\end{align*}
	is an isomorphism.
\end{itemize}

These two definitions of the Thom class coincide, in the sense that the integral Thom class $u_{\cX}$ maps to the de Rahm class $\tau$ under the map 
	\[H^*_{cv} (N\cX;\Z) \xrightarrow{\otimes\R} H^*_{cv} (N\cX;\R) \simeq H^*_{\dR,cv} (N\cX).\]
Moreover, when $\cX$ is compact the cohomology of the normal bundle $N\cX$ with compact vertical support is the same as the cohomology with compact support; therefore, when $\cX$ is compact its Thom class has compact support.

\begin{proposition}\label{prop:thom-class-poincare-dual}
	Consider the well-defined pushforward $i_* : H^*_{\dR,cv} (N_\ep \cX) \allowbreak \to \allowbreak H^*_{\dR}(\cO)$, defined for $\ep$ sufficiently small.  
	If $\cO$ is closed and oriented this map has the following significance: the Poincar\'e dual of $\cX$ can be represented by the pushforward of the Thom class by this map, i.e., \[\PD ([\cX]) = i_* \tau.\]
\end{proposition}

\subsection{Elementary orbifold transversality results}\label{subsec:elementary-orbifold-transversality}

When considering transversality of maps between orbifolds, it is often necessary to use modern techniques such as multisections.
However, when the target space of such a map is a manifold, transversality is easily achievable without such techniques.

\begin{definition}
	Let $\cO$ be an orbifold.  Let $\cM$ be a manifold, and let $\cY \subset \cM$ be a closed embedded submanifold.
	Consider a smooth map $f: \cO \to \cM$ with an associated smooth functor
		$\hat{f} : (O,\bm{O}) \to (M,\bm{M}).$
	We say that $f$ is \textbf{transverse} to $\cY$, written symbolically as $f\pitchfork \cY$, if for every object $x\in \hat{f}^{-1}(Y)\subset O$
		\[
		D\hat{f}_x (T_x O) \oplus T_{\hat{f}(x)} Y = T_{\hat{f}(x)} M.
		\]
\end{definition}

\begin{theorem}\label{thm:perturb-submanifold}
	Consider the above situation of an orbifold $\cO$, a manifold $\cM$, a compact embedded submanifold $\cY \subset \cM$, and a smooth map $f:\cO \to \cM$.
	
	We may perturb $\cY$ so that $f\pitchfork \cY$.  Stated formally, the inclusion map $i: \cY \hookrightarrow \cM$ is homotopic to a smooth inclusion map $i': \cY \hookrightarrow \cM$ such that $f \pitchfork i'(\cY)$.
	In addition, if we assume that $f^{-1}(i(\cY))$ is compact then we may assume that $f^{-1}(i'(\cY))$ is also compact.
\end{theorem}
\begin{proof}
	Choose a cover of the underlying topological space $\cO$ by countably many open sets $\abs{U_k}$ where $U_k \subset O$ are local uniformizers.
	Observe that the property $f\pitchfork \cY$ is equivalent to the property that the collection of smooth maps
		\[\hat{f}|_{U_k} \times i : U_k \times \cY \to \cM\times \cM\]
	are transverse to the diagonal $\Delta \subset \cM \times \cM$.
	
	There exists a smooth family of maps $I: \cY \times B^N  \to \cM$ for some $N$, which is a submersion and such that $I(\cdot,0) = i(\cdot)$. It then follows that the collection of smooth maps
		\[
		\hat{f}|_{U_k} \times I : U_k \times \cY \times B^N \to \cM \times \cM
		\]
	are transverse to the diagonal $\Delta$. So, choose a regular $s\in B^N$ and define $i'(\cdot) := I(\cdot,s)$.
\end{proof}

\begin{theorem}\label{thm:transverse-implies-suborbifold}
	When $f\pitchfork \cY$ the underlying set $f^{-1}(\cY)\subset \cM$ naturally has the structure of an embedded full suborbifold of codimension
		\[
		\codim f^{-1} (\cY) = \codim \cY.
		\]
	Moreover, the normal bundle of $f^{-1}(\cY)$ has fiberwise trivial isotropy action.
\end{theorem}
\begin{proof}
	We check the conditions of Definition~\ref{def:embedded-full-suborbifold}.
	Define a subcategory of $(O,\bm{O})$ as follows:
	\begin{itemize}
		\item define the object set by $X := \hat{f}^{-1}(Y)$,
		\item define the morphism set by $\bm{X} := \hat{f}^{-1}(\bm{Y})$.
	\end{itemize}
	It is immediate that $(X,\bm{X})$ is a full subcategory with saturated object set, hence condition \ref{def:suborbifold-condition-1} is satisfied.
	On the other hand, condition \ref{def:suborbifold-condition-2} follows from the assumption $f\pitchfork \cY$.
	
	Next, observe that the linearization of $\hat{f}$ induces a functor between the tangent bundles
		$
		D\hat{f}: (TO, \bm{TO}) \to (TM, \bm{TM}),
		$
	and moreover, this functor has a well-defined restriction to the normal bundle $(NX,\bm{NX})$ over the suborbifold $(X,\bm{X})$, such that the following diagram commutes:
		\begin{equation*}
			\begin{tikzcd}[column sep = large]
			(NX,\bm{NX}) \arrow[r, "D\hat{f}|_{(NX,\bm{NX})}"] \arrow[d]	& (TM|_Y, \bm{TM}|_{\bm{Y}}) \arrow[r, "\pr_{NY}"] \arrow[d]			& (NY, \bm{NY}) \arrow[d]\\
			(X,\bm{X})	\arrow[r, "\hat{f}|_{(X,\bm{X})}"] 					& (Y, \bm{Y}) \arrow[r,"\id"] & (Y, \bm{Y}).
			\end{tikzcd}
		\end{equation*}
	The transversality assumption implies that $\pr_{N_{\hat{f}(x)}Y} \circ D\hat{f}_x : N_x X \to N_{\hat{f}(x)} Y$ is an isomorphism. However, $N_{\hat{f}(x)} Y$ has trivial isotropy action, hence $N_x X$ must also necessarily have trivial isotropy action.
\end{proof}

\begin{remark}
	Similar observations to the previous theorem were made in \cite{borzellino2012elementary}.
	In our language, they consider a smooth map between orbifolds, $f: \cO \to \cP$
	and prove that the preimage $f^{-1}([y])$ of a regular value $[y] \in \cP$ has the structure of an embedded full suborbifold whose normal bundle has fiberwise trivial isotropy action.
\end{remark}

\subsection[Solving the Steenrod problem for orbifolds]{Solving the {S}teenrod problem for orbifolds}\label{subsec:solving-the-steenrod-problem}

Consider the classifying space $B_k$ with associated universal bundle $E_k$ which together are characterized by the following property: any oriented rank $k$ vector bundle $V$ over a CW-complex $A$ can be obtained as the pullback of $E_k$ via a continuous map to $B_k$, i.e., there exists a continuous map $f: A \to B_k$ such that $f^* E_k = V.$

If we restrict to CW-complexes of dimension less than $n$, we may assume that $B_k$ is the Grassmannian $\hat{G}_k (\R^m)$ for $m$ sufficiently large. The Grassmannian $\hat{G}_k (\R^m)$ consists of oriented real $k$-planes in $\R^m$. 
We may then assume that $E_k$ is the universal quotient bundle $Q$ over $\hat{G}_k(\R^m)$.
Finally recall that $\hat{G}_k (\R^m)$ is a finite-dimensional closed oriented manifold and $Q$ is a rank $k$ oriented vector bundle.

\begin{definition}
	Give $E_k$ a metric, and denote by $D_\ep \subset E_k$ the open $\ep$-disk bundle over $B_k$, consisting of vectors of length $< \ep$. The \textbf{Thom space} $Th(k)$ is defined as the one-point compactification of $D_\ep$, hence the Thom space has underlying set
		\[
		Th(k) := D_\ep \sqcup \{\infty\}.
		\]
\end{definition}

For $*>0$, the pushforward of the inclusion map, $j_*:H^*_c(D_\ep;\Z) \to H^* (Th(k);\Z)$ is an isomorphism \cite[p.~29]{thom1954quelques}.
Without loss of generality we may assume the support of the Thom class of $B_k$ is contained in this $\ep$-disk bundle, i.e., $\supp u_{B_k} \subset D_\ep$. We denote its image under the pushforward of the inclusion by 
	\[U:= j_* u_{B_k}.\]

Via obstruction theory, it is possible to define maps from ($q$-cell approximations of) the Eilenberg--MacLane spaces to the Thom spaces such that the pullback of the class $U$ is a nonzero integer multiple of the distinguished element of an Eilenberg--MacLane space.
Using these techniques, the Thom spaces can be treated as classifying spaces with respect to homology, as the following theorem demonstrates.

\begin{theorem}[{\cite[Thm.~II.25]{thom1954quelques}}] \label{thm:thom-em-space}
	Let $A$ be a CW-complex of finite dimension $n$, and consider a cohomology class $x\in H^k(A;\Z)$.
	There exists a positive nonzero integer $N$ (depending only on $k$ and $n$) such that the following holds.
	There exists a continuous map 
		\[f: A \to Th(k)\]
	such that the pullback $f^*: H^*(Th(k);\Z) \to H^*(A;\Z)$ satisfies the following:
		\[f^*U = N \cdot x.\]
\end{theorem}

We are now able to solve the Steenrod problem for closed orientable orbifolds. Our approach is entirely due to Thom, with modifications made as necessary to deal with the more general nature of orbifolds as opposed to manifolds.

\begin{theorem}\label{thm:steenrod-problem}
	Let $\cO$ be a closed oriented orbifold of dimension $n$.
	\begin{enumerate}[label = \Roman*.]
		\item \label{thm:steenrod-i} Suppose that $\cX\subset \cO$ is a closed oriented embedded full suborbifold of codimension $k$ whose normal bundle $N\cX$ has fiberwise trivial isotropy action. Let $[\cX]\in H_{n-k}(\cO;\Q)$ denote its rational fundamental class.
		Then there exists a continuous map
			\[\tilde{F}: \cO \to Th(k)\]
		such that
			\[\tilde{F}^*U = \PD ([\cX]).\]
		\item \label{thm:steenrod-ii} Consider a rational homology class $x\in H_{n-k}(\cO;\Q)$. Then there exists a positive nonzero integer $N$ such that $N\cdot x$ is the rational fundamental class of a closed oriented embedded full suborbifold $\cX \subset \cO$ of codimension $k$, i.e.,
			\[
			[\cX] = N \cdot x,
			\]
		and moreover, such that the normal bundle $N\cX$ has fiberwise trivial isotropy action.
	\end{enumerate}
\end{theorem}
\begin{proof}[Proof of \ref{thm:steenrod-i}\nopunct] 
In Remark~\ref{rmk:orbit-space-normal-bundle} we observed that underlying topological space $N\cX$ is a vector bundle over the (paracompact) topological space $\cX$. 
Moreover, given orientations on $\cO$ and $\cX$ the bundle $N\cX$ carries an induced orientation.
Hence, there exists a continuous map to the classifying space, $f: \cX \to B_k$, such that $f^*E_k = N\cX$.  We have the following pullback diagram,
	\begin{equation*}
		\begin{tikzcd}
		N\cX \arrow[r, "F"] \arrow[d] & E_k \arrow[d] \\
		\cX \arrow[r, "f"] & B_k,
		\end{tikzcd}
	\end{equation*}
and furthermore, the Thom classes 
	\[u_\cX \in H^k_c(N\cX;\Z), \qquad u_{B_k} \in H^k_c(E_k;\Z) \]
are related by pullback, $F^* u_{B_k} = u_\cX.$

We now replace $E_k$ with the open $\ep$-disk bundle $D_\ep$. 
Consider the open set $N_\ep \cX :=F^{-1}(D_\ep)\subset N\cX$. We may assume that $\supp u_\cX \subset N_\ep \cX$ and $\supp u_{B_k} \subset D_\ep$. For $\ep$ sufficiently small we have inclusion maps 
$i: N_\ep \cX \hookrightarrow \cO$, $j: D_\ep \hookrightarrow Th(k)$ and pushforwards
	\begin{align*}
	i_*:& H^k_c(N_\ep \cX;\Z) \to H^k(\cO;\Z),\\
	j_*:& H^k_c(D_\ep ;\Z) \to H^k(Th(k);\Z).
	\end{align*}
We may extend the restricted map $F: N_\ep \cX \to D_\ep$ to a continuous map $\tilde{F}: \cO \to Th(k)$ as follows:
\begin{itemize}
	\item on $N_\ep \cX$, define $\tilde{F}$ by $F : N_\ep \cX \to D_\ep$,
	\item map all points in $\cO \setminus N_\ep \cX$ to $\{\infty\} \in Th(k)$.
\end{itemize}
We therefore have the following commutative diagram:
\begin{equation*}
	\begin{tikzcd}
	H^k(\cO; \Z) & H^k (Th(k); \Z) \arrow[l, "\tilde{F}^*"] \\
	H^k_c (N_\ep \cX; \Z) \arrow[u, "i_*"] & H^k_c(D_\ep; \Z) \arrow[u, "j_*"] \arrow[l, "F^*"].
	\end{tikzcd}
\end{equation*}
It follows that $\tilde{F}^* U = \tilde{F}^* j_* u_{B_k} = i_* F^* u_{B_k} = i_* u_\cX$.
Switching to de Rahm cohomology, by Proposition~\ref{prop:thom-class-poincare-dual} we see that the Poincar\'e dual of a closed embedded full suborbifold equals the pushforward of the Thom class of its normal bundle, i.e., $\PD([\cX]) = i_* u_\cX$ and therefore
	\[
	\tilde{F}^* U = \PD ([\cX]).
	\] 
The statement for rational coefficients follows from the isomorphism $H_*(\cO;\Q) \allowbreak \otimes \R \simeq H_*(\cO;\R)$.
\end{proof}

\begin{proof}[Proof of \ref{thm:steenrod-ii}\nopunct]
By Theorem~\ref{thm:thom-em-space}, there exists a positive nonzero integer $N$ (depending on only $k=\deg \PD(x)$ and $n= \dim \cO$) such that there exists a continuous map 
	\[f: \cO \to Th(k)\]
such that
	\[f^*U = N \cdot \PD (x).\]

Note that $f^{-1} (B_k) \subset \cO$ is a compact subset (it is a closed subset of the compact topological space $\cO$).
Let $\cO' := \cO \setminus f^{-1} (\infty) $; it is an open subset of $\cO$ and hence carries an induced orbifold structure $(O',\bm{O'})$.

Consider the continuous restricted map
	\[
	f|_{\cO'} : \cO'  \to D_\ep.
	\]
Observe that since $f$ is proper, the restriction $f|_{\cO'}$ is also proper.
By our Whitney approximation theorem~\ref{thm:whitney-approximation} there exists a smooth functor 
	\[
	\hat{g} : (O', \bm{O'}) \to (D_\ep, \bm{D}_\ep)
	\]
such that the induced continuous map $g = \abs{\hat{g}}: \cO' \to D_\ep$ is homotopic to $f|_{\cO'}$, and moreover, such that $g$ is proper. 

By Theorem~\ref{thm:perturb-submanifold}, we may perturb the submanifold $B_k\subset D_\ep$ so that it is transverse to this smooth functor, hence without loss of generality we may assume that $g \pitchfork B_k$.
By Theorem~\ref{thm:transverse-implies-suborbifold}, the set 
	\[
	\cX : = g^{-1} (B_k) \subset \cO'
	\]
has the structure of a closed embedded full suborbifold $(X, \bm{X})$ of codimension $k$ whose normal bundle has fiberwise trivial isotropy action.

Observe that the Thom classes
	\[
	u_\cX \in H^k_c(\cO';\Z), \qquad u_{B_k} \in H^k_c(D_\ep;\Z)
	\]
are related by pullback, $u_\cX = g^* u_{B_k}$. Since $g$ and $f|_{\cO'}$ are homotopic, we have the following commutative diagram
\begin{equation*}
	\begin{tikzcd}
	H^k(\cO; \Z)	& H^k (Th(k) ;\Z) \arrow[l, "f^*"]	\\
	H^k_c(\cO'; \Z) \arrow[u, "i_*"]	& H^k_c( D_\ep ; \Z ) \arrow[l,"g^*"] \arrow[u, "j_*"]
	\end{tikzcd}
\end{equation*}
and therefore
	\[
	\PD ([\cX]) = i_* (u_\cX) = i_* g^* (u_{B_k}) = f^*j_*(u_{B_k}) = f^* U = N \cdot \PD(x).
	\]
Hence $[\cX] = N \cdot x$, as desired.
\end{proof}


\section{Defining polyfold invariants}\label{sec:defining-polyfold-invariants}

Over the past two decades, Hofer, Wysocki, and Zehnder have developed a new approach to
resolving transversality issues that arise in the study of $J$-holomorphic curves in symplectic
geometry called polyfold theory \cites{HWZ1, HWZ2, HWZ3, HWZint, HWZsc, HWZbook}. This approach has been successful in constructing a well-defined Gromov--Witten invariant \cite{HWZGW}.
For a survey of some of the core ideas of polyfold theory, we refer to \cite{ffgw2016polyfoldsfirstandsecondlook}.

\subsection{Polyfolds and weighted branched suborbifolds}\label{subsec:weighted-branched-suborbifolds}

In broad terms, a ``polyfold'' may be viewed as a generalization of an (usually infinite-dimensional) orbifold.
According to the previous mantra, considered as a topological space a {polyfold} is locally homeomorphic to the quotient of an open set of an ``$\ssc$-retract'' by a finite group action.  In formal terms, we offer the following definition.

\begin{definition}[{\cite[Defs.~16.1,~16.3]{HWZbook}}]
	A \textbf{polyfold structure} on a second countable, paracompact, Hausdorff topological space $\cZ$ consists of:
	\begin{itemize}
		\item an ep-groupoid $(Z,\bm{Z})$, whose object and morphism spaces are both M-polyfolds and where the \'etale condition now requires that the source and target morphisms are surjective local $\ssc$-diffeomorphisms,
		\item a homeomorphism $\abs{Z} \simeq \cZ$.
	\end{itemize}
	A \textbf{polyfold} consists of a second countable, paracompact, Hausdorff topological space $\cZ$ together with a Morita equivalence class of polyfold structures $[(Z,\bm{Z})]$ on $\cZ$.
\end{definition}

We will always assume that polyfolds admit $\ssc$-smooth partitions of unity (see \cite[\S~7.5.2]{HWZbook} for further details on this assumption). This assumption is necessary for the proof of Stokes' theorem~\ref{thm:stokes}.

In the present context we will not discuss polyfolds in depth, but instead treat them as ambient topological spaces in which the current objects of study---the weighted branched suborbifolds---sit as subsets.

View $\Q^+:= \Q \cap [0,\infty)$ as an ep-groupoid, having only the identities as morphisms. 
Consider a polyfold, consisting of a polyfold structure $(Z,\bm{Z})$ and an underlying topological space $\cZ$.
Consider a functor $\hat{\theta}: (Z,\bm{Z}) \to \Q^+$ which induces the function $\theta:=\abs{\hat{\theta}} :\cZ \to \Q^+$.
Observe that $\hat{\theta}$ defines a subgroupoid $(S,\bm{S})\subset (Z,\bm{Z})$ with object set
	\[
	S:= \supp (\hat{\theta}) = \{x\in Z\mid \hat{\theta}(x)>0 \}
	\]
and with underlying topological space
	\[
	\cS := \supp (\theta) = \{[x]\in \cZ \mid \theta([x])>0\}.
	\]
Moreover, $(S,\bm{S})$ is a full subcategory of $(Z,\bm{Z})$ whose object set is saturated, i.e., $S= \pi^{-1} (\pi(S))$ where $\pi : Z \to \abs{Z}, x\mapsto [x]$.

\begin{definition}[{\cite[Def.~9.1]{HWZbook}}]\label{def:weighed-branched-suborbifold}
	A \textbf{weighted branched suborbifold structure} consists of a subgroupoid $(S,\bm{S}) \subset (Z,\bm{Z})$ defined by a functor $\hat{\theta} : (Z,\bm{Z}) \to \Q^+$ as above which satisfies the following properties.
	\begin{enumerate}
		\item $\cS \subset \cZ_\infty$.
		\item Given an object $x\in S$, there exists an open neighborhood $U\subset Z$ of $x$ and a finite collection $M_i$, $i\in I$ of finite-dimensional submanifolds of $Z$ (in the sense of \cite[Def.~4.19]{HWZ2}) such that
			\[
			S \cap U= \bigcup_{i \in I}M_i.
			\]
		We require that the submanifolds $M_i$ all have the same dimension. 
		The submanifolds $M_i$ are called \textbf{local branches} in $U$. \label{def:local-branches}
		\item There exist positive rational numbers $w_i$, $i\in I$, (called \textbf{weights}) such that if $y\in S \cap U$, then
		\[\hat{\theta}(y)=\sum_{\{i \in I \mid   y\in M_i\}} w_i.\]
		\item The inclusion maps $\phi_i: M_i\to U$ are proper.
	\end{enumerate}
	We call ${(M_i)}_{i\in I}$ and ${(w_i)}_{i\in I}$ a \textbf{local branching structure}.
\end{definition}

By shrinking the open set $U$ we may assume that the local branches $M_i$ (equipped with the subspace topology induced from $U$) are homeomorphic to open subsets of $\R^n$.  Hence we may assume that a local branch is given by a subset $M_i\subset\R^n$ and an inclusion map $\phi_i : M_i\to U$ where $\phi_i$ is proper and a homeomorphism onto its image.

\begin{definition}\label{def:local-orientation}
	Let $(S,\bm{S})$ be a weighted branched suborbifold structure. Consider an object $x\in S$ and a local branching structure $(M_i)_{i\in I}$, $(w_i)_{i\in I}$ at $x$.
	Suppose moreover that each local branch has an {orientation}, denoted as $(M_i,o_i)$
	
	We define a \textbf{local orientation} at $x$ 
	as the following finite formal sum of weighted oriented tangent planes:
		\[
		\sum_{\{i\in I \mid x\in M_i\}} w_i \cdot T_x (M_i,o_i).
		\]
	We require that this sum is independent of the choice of local branching structure.
	
	An \textbf{orientation} on $(S,\bm{S})$ is defined as a morphism invariant choice of local orientation at every object $x \in S$.
	Explicitly, given a morphism $\phi : x \to y$ there exists a well-defined tangent map $T\phi : T_xZ \to T_yZ$.
	The image of a finite formal sum of weighted oriented tangent planes under this map is again a finite formal sum of weighted oriented tangent planes, and hence we require invariance of the local orientations in the following sense:
	\[
	\sum_{\{j\in I' \mid y\in M'_j\}} w'_j \cdot T_y (M'_j,o'_j) = \sum_{\{i\in I \mid x\in M_i\}} w_i \cdot T\phi_* (T_x (M_i,o_i)).
	\]
\end{definition}

A \textbf{weighted branched suborbifold structure with boundary} consists of a subgroupoid $(S,\bm{S})\subset (Z,\bm{Z})$ defined identically to Definition~\ref{def:weighed-branched-suborbifold} except we allow the possibility that the local branches are manifolds with boundary.
A \textbf{local orientation} at an object $x\in S$ is again defined as in Definition~\ref{def:local-orientation} as a finite formal sum determined by orientations of the local branches, and likewise an \textbf{orientation} is a morphism invariant choice of local orientations.

\subsection{Polyfold invariants as branched integrals}\label{subsec:branched-integrals}

We recall the branched integration theory on compact oriented weighted branched suborbifolds, as originally developed in \cite{HWZint}.

\begin{definition}[{\cite[Def.~4.9]{HWZbook}}]
Let $\cZ$ be a polyfold with an associated polyfold structure $(Z,\bm{Z})$.
The vector space $\Omega^k (Z)$ consisting of $\ssc$-differential $k$-forms 
	\[
	\ww:\bigoplus^k_{n=1} TZ\rightarrow \R.
	\]
is the set of $\ssc$-smooth maps defined on the Whitney sum of the tangent of the object space, which are linear in each argument and skew-symmetric.	
Moreover, we require that the maps $\ww$ are morphism invariant in the following sense: for every morphism $\phi: x\to y$ in $\bm{Z}_1$ with tangent map $T\phi:T_xZ\rightarrow T_yZ$ we require that
	\[
	(T\phi)^*\ww_y=\ww_x.
	\]
\end{definition}

Recall the definition of $\cZ^i$ as the shifted polyfold with shifted polyfold structure $(Z^i,\bm{Z}^i)$ (see \cite[pp.~21--22]{HWZbook}).
Via the inclusion maps $\cZ^i \hookrightarrow \cZ$ we may pullback a $\ssc$-differential $k$-form $\ww$ in $\Omega^k(Z)$ to $\Omega^k(Z^i)$, obtaining a directed system
	\[
	\Omega^k(Z) \to \cdots \to \Omega^k(Z^i) \to \Omega^k(Z^{i+1}) \to \cdots,
	\]
we denote by $\Omega^k_\infty (Z)$ the direct limit of this system.
As defined in \cite[p.~149]{HWZbook} there exists an \textbf{exterior derivative}
	\[
	d:\Omega^*(Z^{i+1}) \to \Omega^{* +1}(Z^i)
	\]
such that the composition $d\circ d = 0$.
The exterior derivative commutes with the inclusion maps $Z^i \hookrightarrow Z^{i+1}$ and hence induces a map
	\[
	d:\Omega^*_\infty(Z) \to \Omega^{* +1}_\infty(Z)
	\]
which also satisfies $d\circ d =0$.

\begin{theorem}[{\cite[Thm.~9.2]{HWZbook}}]
	\label{def:branched-integral}
	Let $\cZ$ be a polyfold with polyfold structure $(Z,\bm{Z})$.
	Given a $\ssc$-smooth differential form $\ww\in \Omega^n_\infty (Z)$ and an $n$-dimensional compact oriented weighted branched suborbifold $\cS\subset \cZ$.
	
	Then there exists a well-defined \textbf{branched integral}, denoted as
		$\int_{\cS} \ww,$
	which is partially characterized by the following property. 
	Consider a point $[x]\in \cS$ and a representative $x\in S$ with isotropy group $\bm{G}(x)$. Let $(M_i)_{i\in I}$, $(w_i)_{i\in I}$, $(o_i)_{i\in I}$ be an oriented local branching structure at $x$ contained in a $\bm{G}(x)$-invariant open neighborhood $U\subset Z$ of $x$.
	Consider a $\ssc$-smooth differential form $\ww\in \Omega^n_\infty (Z)$ and suppose that $\abs{\supp \ww} \subset \abs{U}$.
	Then
		\[
		\int_{\cS} \ww = \frac{1}{\sharp \bm{G}^\text{eff}(x)} \sum_{i\in I} w_i \int_{(M_i,o_i)} \ww,
		\]
	where $\sharp \bm{G}^\text{eff}(x)$ is the order of the effective isotropy group and $\int_{(M_i,o_i)} \ww$ is the usual integration of the differential $n$-form $\ww$ on the oriented $n$-dimensional manifold $M_i$.
\end{theorem}

\begin{remark}[Computation of the branched integral]
	The proof of \cite[Thm.~9.4]{HWZbook} shows how to compute the branched integral, relying on the assumption that the polyfold $\cZ$ admits $\ssc$-smooth partitions of unity.
	
	Cover the underlying compact topological space $\cS$ with finitely many open sets $\abs{U_{x_j}}\subset \cZ$, $j\in \{1,\ldots,k\}$, where $U_{x_j}\subset Z$ are $\bm{G}(x_j)$-invariant local uniformizers.  Moreover, suppose that $S\cap U_{x_j} = \cup_{i\in I_j} M_i$ where $(M_i,o_i)$ are oriented local branches with associated weights $w_i$.
	Denote by $U_{x_j}^*$ the saturations of the sets $U_{x_j}$ in $Z$, i.e., $U_{x_j}^* = \pi^{-1}(\pi(U_{x_j}))$.
	We may find a saturated open set $U^*_0$ such that
		\[Z= U^*_0 \cup U^*_{x_1} \cup \cdots \cup U^*_{x_k}, \qquad
		\cS \subset \bigcup_{j=1}^k	\abs{U^*_{x_j}\setminus \overline{U^*_0}}.\]
	The polyfold $\cZ$ admits $\ssc$-smooth partitions of unity, hence there exist morphism invariant $\ssc$-smooth functions $\beta_0,\beta_1,\ldots,\beta_k :Z \to [0,1]$ which satisfy $\supp \beta_0 \subset U^*_0$, $\supp \beta_j\subset U^*_{x_j}$.  It follows that $\sum_{j=1}^k \beta_j =1$ on $S$ and $\abs{\supp \beta_j} \subset \abs{U^*_{x_j}}$.
	
	We may now write
	\[
	\int_{\cS} \ww 
	= \sum_{j=0}^k \int_{\cS} \beta_j \cdot \ww 
	= \sum_{j=1}^k \frac{1}{\sharp \bm{G}^\text{eff}(x_j)} \sum_{i\in I_j} w_i \int_{(M_i,o_i)} \beta_j \cdot \ww
	\]
	(we drop the term $j=0$ in the second equality since $\abs{\supp \beta_0} \cap \cS = \emptyset$).
\end{remark}

\begin{theorem}[Stokes' theorem, {\cite[Thm.~9.4]{HWZbook}}]
	\label{thm:stokes}
	Let $\cZ$ be a polyfold with polyfold structure $(Z,\bm{Z})$ which admits $\ssc$-smooth partitions of unity.
	Let $\cS$ be an $n$-dimensional compact oriented weighted branched suborbifold, and let $\partial \cS$ be its boundary with induced weights and orientation.  Consider a $\ssc$-differential form $\omega\in \Omega^{n-1}_\infty (Z)$.
	Then 
	\[
	\int_{\cS} d\omega   = \int_{\partial \cS} \omega.
	\]
\end{theorem}

As discussed in the introduction, given a $\ssc$-smooth map 
	$
	f:\cZ \to \cO
	$
from a polyfold $\cZ$ to an orbifold $\cO$, the \textbf{polyfold invariant} is the homomorphism obtained by pulling back a de Rahm cohomology class from the orbifold and taking the branched integral over a perturbed zero set:
	\[
	H^*_{\dR} (\cO) \to \R,	\qquad \ww \mapsto \int_{\cS(p)} f^*\ww.
	\]
Using Stokes' theorem~\ref{thm:stokes}, we see that this homomorphism does not depend on the choice of abstract perturbation used to obtain the compact oriented weighted branched orbifold $\cS(p)$.

\subsection{Polyfold invariants as intersection numbers}\label{subsec:intersection-numbers}

Compact oriented weighted branched suborbifolds possess suitable notions smooth maps and of oriented tangent spaces; as such, we can generalize appropriate notions of transversal intersection and intersection number.

Let $\cS$ be a compact oriented weighted branched suborbifold.
Let $\cO$ be an oriented orbifold, and let $\cX$ be a closed embedded full suborbifold whose normal bundle has fiberwise trivial isotropy action.
Consider a smooth map $f:\cS\to\cO$ with an associated smooth functor $\hat{f}: (S,\bm{S}) \to (O,\bm{O}).$

\begin{definition}
	We say that $f$ is \textbf{transverse} to $\cX$, written symbolically as $f\pitchfork \cX$, if for every object $x\in \hat{f}^{-1}(X)$
		\[
		D\hat{f}_x (T_x M_i) \oplus T_{\hat{f}(x)} X = T_{\hat{f}(x)} O
		\]
	for every local branch $M_i$ of a given local branching structure $(M_i)_{i\in I}$ at $x$.
\end{definition}

\subsubsection{Achieving transversality}
Transversal intersection is a generic property, and may be obtained via either of the following two propositions.

\begin{proposition}[Transversality through perturbation of the embedded full suborbifold]
	\label{prop:transversality-perturbation-full-suborbifold}
	We may perturb $\cX$ so that $f \pitchfork \cX$. Stated formally, the inclusion map $i : \cX \hookrightarrow \cO$ is homotopic to a smooth inclusion map $i' : \cX \hookrightarrow \cO$ such that $f \pitchfork i'(\cX)$.
\end{proposition}
\begin{proof}
Consider the normal bundle $N\cX$ and let $k := \text{rank } N\cX$.
There exists a finite collection of local uniformizers $U_i \subset X$ centered at objects $x_i$ such that the open sets $\abs{U_i}$ cover the underlying topological space $\cS$.
We may furthermore assume that these local uniformizers also give local trivializations of the normal bundle $P:N\cX\to \cX$, i.e.,
	\[
	\hat{P}^{-1}(U_i) \simeq U_i \times \R^k.
	\]

We claim that there exists a smooth map between orbifolds 
	\[
	F : \cX \times B_\ep^N \to N\cX
	\]
which is a submersion and such that $F(\cdot, 0)$ is the zero section.
Here we consider the open disc $B_\ep^N = \{x\in \R^N \mid \abs{x}<\ep\}$ as a manifold with the trivial orbifold structure (see Definition~\ref{def:manifold-as-trivial-orbifold}).
This functor may be constructed by taking the zero section plus a finite sum of parametrized sections which span the fiber of the normal bundle at every point. This is not normally possible without using multisections---however we note that the assumption that $N\cX$ has fiberwise trivial isotropy action implies that $\bm{G}(x)$ acts trivially on the second factor of $U_i \times \R^k$. 
Therefore, we may locally define parametrized sections which are invariant under the $\bm{G}(x)$-action as follows: 
	\[
	U_i \times B^k_\ep \to \R^k,\qquad (y,t_1,\ldots,t_k) \mapsto \beta(y)  \sum_{j=1}^{k} t_j e_j
	\]
where $e_j\in \R^k$ are a basis of the unit vectors and where $\beta: U_i \to [0,1]$ is a $\bm{G}(x)$-invariant bump function. This locally defined expression may then be extended to a globally defined, morphism invariant, parametrized section of the bundle $NX \to X$ which induces a smooth parametrized section of $N\cX \to \cX$.

For $\ep$ sufficiently small, there exists a sufficiently small open neighborhood $N_\ep\cX$ of the zero section $\cX \to N_\ep \cX$ such that there is a well-defined map between orbifolds
	$i : N_\ep \cX \to \cO$
which is a homeomorphism onto its image when considered with respect to the underlying topological spaces, and which is a local diffeomorphism when considered with respect to the orbifold structures.  It therefore follows that for sufficiently small $\ep' \leq \ep$ we have a well-defined smooth submersion given by the composition
	\[
	I:= i\circ F : \cX \times B_{\ep'}^N \to N_{\ep'}\cX \to \cO.
	\]

Consider the map $f \times I: \cS \times \cX \times B_{\ep'}^N \to \cO \times \cO$, and observe that this map is necessarily transverse to the diagonal of the object space $O$, i.e., $\Delta \subset O\times O$. Hence for any object $(x,y,s) \in \hat{f}^{-1}(\Delta)$
	\[
	D\hat{f}_x (T_x M_i) \oplus D\hat{I}_{(y,s)} (T_y X \times \R^N)  = T_{(\hat{f}(x),\hat{I}(y,s))} \Delta
	\]
where $(M_i)_{i\in I}$ is a local branching structure at the object $x\in S$.
Choose a finite cover of the underlying topological space $\cS$ by open sets of the form $\abs{\cup M_i}$.
It follows that the finite collection of smooth maps
	\[
	\hat{f}|_{M_i} \times \hat{I}|_{U_j\times B_{\ep'}^N} : M_i \times U_j \times B_{\ep'}^N \to O\times O
	\]
are all transverse to the diagonal $\Delta \subset O\times O$.
Choose a common regular value $s\in B_{\ep'}^N$ and define the smooth inclusion map $i'(\cdot) := I(\cdot,s) : \cX \hookrightarrow \cO$; it follows that the collection of smooth maps
	\[
	\hat{f}|_{M_i} \times \hat{i'}|_{U_j} : M_i \times U_j \to O\times O
	\]
are transverse to the diagonal.
This implies that $f\pitchfork i' (\cX)$, as desired.
\end{proof}

In some situations, it is undesirable to perturb the full suborbifold.
For example, the full suborbifold might represent fixed constraints we wish to impose on the perturbed solution set.
In these situations, we wish to achieve transversality by choice of a suitable generic abstract perturbation.
The following proposition requires familiarity with some of the abstract machinery of polyfold theory, specifically of the construction of regular $\ssc^+$-multisections in \cite[Ch.~15]{HWZbook}.

\begin{proposition}[Transversality through construction of a regular perturbation]
	\label{prop:transversality-regular-perturbation}
	Consider a \textit{polyfold Fredholm problem} consisting of a strong polyfold bundle $\cW$ over a polyfold $\cZ$ together with a $\ssc$-smooth proper Fredholm section $\delbar$:
		\[
			\begin{tikzcd}
			\cW \arrow[r,"P"'] & \cZ. \arrow[l, bend right, swap, "\delbar"]
			\end{tikzcd}
		\]
	Let $\cO$ be an oriented orbifold, and let $\cX$ be a closed embedded full suborbifold whose normal bundle has fiberwise trivial isotropy action.
	Consider a $\ssc$-smooth map $f:\cZ\to\cO$ with an associated $\ssc$-smooth functor $\hat{f}: (Z,\bm{Z}) \to (O,\bm{O}).$
	Suppose moreover that $f$ is a \textit{submersion} on $\cZ_\infty \subset \cZ$, i.e., for any object $x \in Z_\infty$ we have $D\hat{f}_x (T_x Z) = T_{\hat{f}(x)} O$.
	
	Then there exists a choice of regular perturbation such that the map $f$ restricted to the perturbed solution set defined by this perturbation is transverse to the embedded full suborbifold $\cX \subset \cO$.
	
	Restated formally, there exists a choice of regular $\ssc^+$-multisection
		$p: \cW \to \Q^+$
	whose perturbed solution set 
		$\cS(p)= \{	[x]\in \cZ \mid p(\delbar([x]))>0	\}$
	has the structure of a compact oriented weighted branched suborbifold, and such that $f$ restricted to $\cS(p)$ is transverse to $\cX$,
		\[f|_{\cS(p)} \pitchfork \cX.\]
		
    Furthermore, the subset $\cS(p)\cap f^{-1}(\cX)$ also has the the structure of a compact oriented weighted branched suborbifold.
\end{proposition}
\begin{proof}

Our proof follows the general position argument of \cite[Thm.~15.4]{HWZbook}.
The basic idea is that, in addition to filling in the cokernel of the linearization of $\hdelbar$ via parametrized $\ssc^+$-sections, we use the fact that $f$ is a submersion to fill in additional vectors in order to obtain a transversal perturbation of $\delbar \times f$.

Denote the unperturbed zero set by $\cS(\delbar) := \{	[x]\in \cZ \mid \delbar([x])=0	\}$; it has an associated subgroupoid $(S(\hdelbar),\bm{S}(\hdelbar))$. Moreover, recall that $\cS(\delbar)\subset \cZ_\infty$ and that $\cS(\delbar)$ is compact.
As additional data, fix a pair $(N,\cU)$ which controls the compactness of $\delbar$ (see \cite[Def.~15.4]{HWZbook}).

\noindent\emph{Local construction of a parametrized $\ssc^+$-multisection.}
Consider a point $[x_0]\in \cS(\delbar)$ and let $x_0\in S(\hdelbar)$ be a representative with isotropy group $\bm{G}(x_0)$. Let $U\subset Z$ be a $\bm{G}(x_0)$-invariant open neighborhood of $x_0$, and moreover let $V\subset U$ be a $\bm{G}(x_0)$-invariant open neighborhood of $x_0$ such that $\overline{V} \subset U$.

We may choose smooth vectors $v_1,\ldots, v_m \in W_{x_0}$ such that
	\[
	\text{span}\{v_1,\ldots,v_m\} \oplus D \hdelbar_{x_0} (T_{x_0} Z) = W_{x_0}.
	\]
By assumption we may also choose vectors $a_1,\ldots, a_n \in T_{x_0} Z$ such that
	\[
	\text{span} \{D\hat{f}_{x_0} (a_1), \ldots,D\hat{f}_{x_0} (a_n)\}  = T_{\hat{f}(x_0)} O.
	\]
Without loss of generality we may assume that these vectors are smooth;
to see this, observe that smooth vectors are dense in $T_{x_0} Z$ and $D\hat{f}_{x_0}$ is continuous, hence a small perturbation $a'_1,\ldots, a'_n$ will yield a spanning set $\{D\hat{f}_{x_0} (a'_1), \ldots,D\hat{f}_{x_0} (a'_n)\}$ of the finite-dimensional vector space $T_{\hat{f}(x_0)} O$.

Let $v_{m+i} : = D \hdelbar_{x_0}(a_i) \in W_{x_0}$.
For each smooth vector $v_1,\ldots, v_{m+n}$ we may use \cite[Lem.~5.3]{HWZbook} to define $\ssc^+$-sections $s_i : U \to W$ such that
	\begin{itemize}
		\item $s_i= 0$ on $U\setminus V$,
		\item $s_i(x_0) = v_i$.
	\end{itemize}
Furthermore, to ensure that the resulting multisection is controlled by the pair $(N,\cU)$ we require that
	\begin{itemize}
		\item $N[s_i] \leq 1,$
		\item $\text{supp}(s_i)\subset \cU$.
	\end{itemize}
Define a parametrized $\ssc^+$-section as follows:
	\[
	s_{x_0}:	U \times B_\ep^{m+n}	\to 		W,
	\qquad (x,t) \mapsto	\sum_{i=1}^{m+n} t_i s_i (x).
	\]
Now observe that the function defined by
	\[
	U \times B_\ep^{m+n} 	\to W \times O, 
	\qquad (x, t) \mapsto (\hdelbar (x) - s_{x_0} (x,t), \hat{f}(x))
	\]
is $\ssc$-Fredholm and its linearization, projected to the fiber $W_{x_0} \times T_{\hat{f}(x_0)} O$, is surjective in a neighborhood of $x_0$ by construction. (This projection is well-defined in a neighborhood of $x_0$ in local $\ssc$-coordinates.)

For every element $g$ of the isotropy group $\bm{G}(x_0)$ we define a $\ssc^+$-section $g * s_{x_0} : U \times B_\ep^{m+n} \to W$ by the following equation:
	\[
	\left(g * s_{x_0} \right) (x,t) := \mu (\Gamma (g, g^{-1}*x), s_{x_0} (g^{-1}* x,t))
	\]
where $\Gamma$ is the natural representation of $\bm{G}(x_0)$ on $U$.
Note that due to the morphism invariance of $\hdelbar$ and $\hat{f}$ the linearization of the function 
	\begin{align*}
	U \times B_\ep^{m+n} 	&\to W \times O \\
	(x, t) 					&\mapsto (\hdelbar (x) - g*s_{x_0} (x,t), \hat{f}(x)),
	\end{align*}
projected to the fiber $W_{x_0} \times T_{\hat{f}(x_0)} O$ is surjective in a neighborhood of $x_0$.

The collection of parametrized $\ssc^+$-sections $\{g * s_{x_0}\}_{g\in \bm{G}(x_0)}$ is $\bm{G}(x_0)$-in\-vari\-ant under the obvious action; the collection $\{g * s_{x_0}\}_{g\in \bm{G}(x_0)}$ with associated weights ${1}/{\sharp \bm{G}(x_0)}$ together give the local section structure for a globally defined $\ssc^+$-multi\-section functor
	\[
	\hat{\Lambda}_0 : W \times B_\ep^{m+n} \to \Q^+.
	\]

\noindent\emph{Global construction of a parametrized $\ssc^+$-multisection.}
Using the above construction, at any point $[x_0]\in \cS(\delbar)$ we have defined a parametrized $\ssc^+$-multisection functor $\hat{\Lambda}_0$ which satisfies the following.
There exists an open neighborhood $[x_0] \in \cU \subset \cZ$ such that: for any point $[x]\in \cU$ and for any parametrized local section structure $\{s_i\}_{i\in I}$ of $\hat{\Lambda}_0$ defined on an open neighborhood $V$ of a representative $x$, the linearization of the function 
	\begin{align*}
	V \times B_\ep^{n+m} 	&\to W \times O \\
	(x, t) 					&\mapsto (\hdelbar (x) - s_i(x,t), \hat{f}(x))
	\end{align*}
projected to the fiber $W_{x_0} \times T_{\hat{f}(x_0)} O$, is surjective.

We may cover the compact topological space $\cS(\delbar)$ by a finite collection of such sets $\cU_j$; it follows that the finite sum of $\ssc^+$-multisections
	\[
	\Lambda:= \bigoplus_j \Lambda_j : \cW \times B_\ep^N \to \Q^+
	\]
has the property that for any point $[x] \in \cZ$ with $\Lambda \circ \delbar ([x])>0$ and for any parametrized local section structure $\{s_i\}_{i\in I}$ of $\Lambda$ defined on an open neighborhood $V$ of a representative $x$, the linearization of the function 
	\begin{align*}
	V \times B_\ep^N 	&\to W \times O \\
	(x, t) 					&\mapsto (\hdelbar (x) - s_i(x,t), \hat{f}(x))
	\end{align*}
projected to the fiber $W_{x} \times T_{\hat{f}(x)} O$, is surjective.
In the terminology of \cite[Def.~15.2]{HWZbook}, this we say that the $\ssc^+$-multisection $\Lambda$ is a ``transversal perturbation'' of the map $\delbar \times f : \cZ \to \cW \times \cO$.

Furthermore for $\ep$ sufficiently small, for any fixed $t_0 \in B_\ep^N$ the $\ssc^+$-multisection $\Lambda (\cdot, t_0)$ is controlled by the pair $(N,\cU)$, i.e.,
	\begin{itemize}
		\item $N[\Lambda(\cdot,t_0)] \leq 1$,
		\item $\text{dom-supp}(\Lambda(\cdot,t_0)) \subset \cU$.
	\end{itemize}

\noindent\emph{The thickened solution set without and with constraints.}
The polyfold implicit function theorem \cite[Thm.~3.14]{HWZbook} implies that the thickened solution set without constraints,
	\[
	\cS (\Lambda; B_\ep^N): = \{	([x],t) \in \cZ \times B_\ep^N \mid \Lambda(\delbar([x]),t)>0\},
	\]
has the structure of a weighted branched suborbifold (see \cite[Thm.~15.2]{HWZbook} for further details).
Furthermore, the thickened solution set with constraints,
	\[
	\cS (\Lambda; B_\ep^N)\cap f^{-1}(\cX) = \{	([x],t) \in \cZ \times B_\ep^N \mid \Lambda(\delbar([x]),t)>0, f([x])\in \cX	\},
	\]
also has the structure of a weighted branched suborbifold; this follows from the same argument, with the polyfold implicit function theorem applied to a local coordinate representation of the transversal perturbation of the map $\delbar \times f : \cZ \to \cW \times \cO$.

By Sard's theorem, we can find a common regular value $t_0 \in B_\ep^N$ of the projections $\cS (\Lambda; B_\ep^N) \to B_\ep^N$ and $\cS (\Lambda; B_\ep^N) \cap f^{-1}(\cX) \to B_\ep^N$;
the perturbed solution set without constraints
$\cS (\Lambda(\cdot, t_0)): = \{	[x] \in \cZ \mid \Lambda(\delbar([x]),t_0)>0\}$
and the perturbed solution set with constraints
$\cS (\Lambda(\cdot,t_0))\cap f^{-1}(\cX) = \{	[x] \in \cZ \mid \Lambda(\delbar([x]),t_0)>0, f([x])\in \cX	\},$
both have the structure of weighted branched suborbifolds.
We may moreover assume that $\Lambda(\cdot,t_0)$ is controlled by the pair $(N,\cU)$, and hence by \cite[Lem.~4.16]{HWZ3} the underlying topological space $\cS (\Lambda(\cdot, t_0))$ is compact.  This implies $\cS (\Lambda(\cdot,t_0))\cap f^{-1}(\cX)$ is also compact (as it is a closed subset of $\cS (\Lambda(\cdot, t_0))$).

\noindent\emph{Transversality of the perturbed solution set and the suborbifold.}
All that remains to complete the proof of the theorem is to demonstrate that $f|_{\cS(\Lambda (\cdot, t_0))} \pitchfork \cX$.
By definition, this is satisfied if for every point $[x]\in \cS (\Lambda(\cdot,t_0))\cap f^{-1}(\cX) \subset \cS(\Lambda (\cdot, t_0))$ and for any local branch $M_i$ of a local branching structure $(M_i)_{i\in I}$ at a representative $x$ (with respect to the weighted branched suborbifold $\cS(\Lambda (\cdot, t_0))$) we have:
	\[
	D\hat{f}_x(T_x M_i) \oplus T_{\hat{f}(x)} X = T_{\hat{f}(x)} O.
	\]
To see that this is true, consider a local section structure $(s_i)_{i\in I}$ of $\hat{\Lambda}(\cdot,t_0)$ at $x$ which defines this local branching structure; hence
    \[
    M_i = (\hdelbar - s_i)^{-1}(0)
    \]
and moreover
    \begin{equation}\label{eq:tangent-TM}
    T_x M_i = \ker \pr_{W_x} D (\hdelbar - s_i)_x.
    \end{equation}
The local branching structure at $x$ for $\cS (\Lambda(\cdot,t_0))\cap f^{-1}(\cX)$ is then given by $M_i':= (M_i \cap \hat{f}^{-1}(X))$; moreover
    \begin{equation}\label{eq:tangent-TM-prime}
    T_x M_i' = \ker \pr_{W_x} D (\hdelbar - s_i)_x \cap \ker \pr_{N_{\hat{f}(x)}X} D\hat{f}_x.
    \end{equation}
We may decompose $T_x Z$ as
    \[
    T_xZ = E \oplus \underbrace{F \oplus T_x M_i'.}_{= T_x M_i}
    \]
Since $\pr_{W_x} D (\hdelbar - s_i)_x : T_xZ \to W_x$ is surjective, it follows from \eqref{eq:tangent-TM} that its restriction to $E$ is an isomorphism.
Since $\pr_{W_x} D (\hdelbar - s_i)_x \times \pr_{N_{\hat{f}(x)}X} D\hat{f}_x : T_x Z \to W_x \times N_{\hat{f}(x)}X$ is surjective, it follows from \eqref{eq:tangent-TM-prime} that its restriction to $E\oplus F$ is an isomorphism.
From these two observations it necessarily follows that $\pr_{N_{\hat{f}(x)}X} D\hat{f}_x |_F : F \to N_{\hat{f}(x)}X$ is also an isomorphism.
Since $N_{\hat{f}(x)}X$ and $T_{\hat{f}(x)} X$ together span $T_{\hat{f}(x)} O$ we see that $D\hat{f}_x(T_x M_i) \oplus T_{\hat{f}(x)} X = T_{\hat{f}(x)} O$ as desired.    
\end{proof}

Consider again the setup described at the start of the section, \S~\ref{subsec:intersection-numbers}.

\begin{lemma}\label{lem:neighborhood}
	Given an open neighborhood $\cU \subset \cS$ of $f^{-1}(\cX)$ there exists an open neighborhood $\cV \subset \cO$ of $\cX$ such that
		$
		f^{-1} (\cV) \subset \cU.
		$
\end{lemma}
\begin{proof}
	This follows from basic point-set topology.
	The set $\cS \setminus \cU$ is closed, hence compact; it follows that $f(\cS \setminus \cU)$ is also compact.
	Observe that $f(\cS \setminus \cU) \cap \cX = \emptyset$.
	
	The underlying topological space $\cO$ is by definition paracompact and Hausdorff; this implies that it is a normal topological space.  We may therefore separate the disjoint closed sets $f(\cS \setminus \cU)$ and $\cX$, which implies we can find an open neighborhood $\cV$ of $\cX$ such that $f(\cS \setminus \cU) \cap \cV = \emptyset$. The claim then follows from the fact that $(\cS \setminus \cU) \cap f^{-1} (\cV) = \emptyset$.
\end{proof}

\begin{lemma}\label{lem:transverse-intersection}
	Without loss of generality, we may assume that $f\pitchfork \cX$. Suppose that $\dim \cS + \dim \cX = \dim \cO$.
	\begin{enumerate}
		\item \label{lem:1} The set $f^{-1}(\cX) \subset \cS$ consists of isolated points. By compactness of $\cS$, it follows that $f^{-1}(\cX)$ is a finite set of points.
		\item \label{lem:2} Consider a point $[x] \in f^{-1}(\cX)$. Consider a local branching structure $(M_i)_{i\in I}$ (with respect to $\cS$) at a representative $x\in S$, contained in a $\bm{G}(x)$-invariant open neighborhood $U$.
		Then a given local branch $M_i$ is fixed only by the identity $\id \in \bm{G}^\text{eff}(x)$, i.e.,
			\[
			g * M_i \neq M_i \qquad \text{for all } g\in \bm{G}^\text{eff}(x), g\neq \id.
			\]
		Moreover, this implies that $\sharp \bm{G}^\text{eff}(x)$ divides $\sharp I$.
	\end{enumerate}
\end{lemma}
\begin{proof}[Proof of (\ref{lem:1})]
	Consider a point $[x] \in f^{-1}(\cX) \subset \cS$ and consider a local branching structure $(M_i)_{i\in I}$ (with respect to $\cS$) at a representative $x$.
	For every given local branch $M_i$, the dimension assumption and the fact that 
		\[
		D\hat{f}_x(T_x M_i) \oplus T_{\hat{f}(x)} X = T_{\hat{f}(x)} O
		\]
	together imply that $x \in M_i \cap \hat{f}^{-1}(X)$ is isolated, i.e., we can find open neighborhoods $x \in U_i \subset M_i$ such that $U_i \cap \hat{f}^{-1}(X) = \{x\}$.
	Since the topology on each branch $M_i$ is the same as the subspace topology induced from $Z$, we can find a $\bm{G}(x)$-invariant open set $U\subset Z$ such that $U\cap M_i \subset U_i$ for all $i\in I$, and which has the property that $U \cap (\cup M_i) \cap \hat{f}^{-1}(X) = \{x\}$.
	The set $U\cap (\cup M_i)$ is an open subset of the space $S$; since $\pi : S \to \cS, x \mapsto [x]$ is an open map it follows that $\abs{U\cap (\cup M_i) } \subset \cS$ is an open set.
	
	We claim that $\abs{U\cap (\cup M_i) } \cap f^{-1}(\cX) = \{[x]\}$.  To see this, suppose there were a point $[y] \in \abs{U\cap (\cup M_i) } \cap f^{-1}(\cX)$, $[y]\neq [x]$; we could then find objects $y_1 \in U\cap (\cup M_i)$, $y_1 \mapsto [y]$ and $y_2 \in \pi^{-1} (f^{-1}(\cX)) = \hat{f}^{-1}(X)$, $y_2 \mapsto [y]$. But since $\hat{f}^{-1}(X)$ is saturated it follows that $y_1 \in \hat{f}^{-1}(X)$, which contradicts the fact that $U \cap (\cup M_i) \cap \hat{f}^{-1}(X) = \{x\}$.
	
	We have shown that the point $[x] \in f^{-1}(\cX)$ is isolated; since $[x]$ was arbitrary, the proof is complete.
\end{proof}
\begin{proof}[Proof of (\ref{lem:2})]
	Let $y := \hat{f}(x)\in O$, and let $U(y)\subset O$ be a $\bm{G}(y)$-invariant open neighborhood. The functor $\hat{f}$ induces a group homomorphism between the isotropy groups, $\hat{f}: \bm{G}(x) \to \bm{G}(y)$ and an equivariant map
		\[
		\hat{f} : \cup_{i\in I} M_i \to U(y)
		\]
	with respect to these isotropy groups.
	The tangent planes $\cup_{i\in I} T_{x} M_i$ and $T_{y} O$ carry an induced action by the isotropy groups; the linearization 
		\[
		D\hat{f}_{x} : \cup_{i\in I} T_{x} M_i \to T_{y} O
		\]
	is equivariant with respect to this action.  The projection to the normal bundle
		$
		\pr_{N_{y}X} : T_{y} O \to N_{y}X
		$
	is also equivariant.  Recall that by assumption on the suborbifold $\cX$, the isotropy group $\bm{G}(y)$ acts trivially on the fiber $N_{y}X$.
	
	The composition
		\[
		\pr_{N_{y}X} \circ D\hat{f}_{x} : \cup_{i\in I} T_{x} M_i \to N_{y}X
		\]
	is equivariant.  The fact that $\dim M_i = \rank NX$ together with the transversality assumption implies that the map 
		\begin{equation}\label{eq:equivariant-map}
		\pr_{N_{y}X} \circ D\hat{f}_{x} |_{T_x M_i} : T_{x} M_i \to N_{y}X
		\end{equation}
	is an isomorphism. Consider the induced action on the tangent plan $T_x M_i$ by any $g\in \bm{G}^\text{eff} (x)$, $g\neq \id$ with $g\ast M_i = M_i$; the map \eqref{eq:equivariant-map} must be equivariant with respect to this action. However, since $\bm{G}(y)$ acts trivially on the fiber $N_{y}X$ this contradicts the fact that \eqref{eq:equivariant-map} is an isomorphism. Hence no such $g\in \bm{G}^\text{eff} (x)$, $g\neq \id$ exists, and the claim is proven.
\end{proof}

\begin{definition}\label{def:intersection-number}
	If $\dim \cS + \dim \cX = \dim \cO$ we define the \textbf{intersection number}
	$f|_{\cS} \cdot \cX$ by the equation
	\[
	f|_{\cS} \cdot \cX := \sum_{[x]\in f^{-1}(\cX)} \left(\frac{1}{\sharp \bm{G}^\text{eff}(x)} \sum_{\{ i\in I \mid x\in M_i \}} \sign (x;i) w_i \right)
	\]
	where $x$ is a representative of $[x]$, and $(M_i)_{i\in I}$ are local branches at $x$ with weights $(w_i)_{i\in I}$ and orientations $(o_i)_{i\in I}$. The sign $\sign(x;i)=\pm 1$ is positive if $D\hat{f}_x (T_x M_i) \oplus T_{\hat{f}(x)} X$ has the same orientation as $T_{\hat{f}(x)} O$ and negative if the orientations are opposite.
	
	If $\dim \cS + \dim \cX \neq \dim \cO$ we define the intersection number to be zero, 
		\[f|_{\cS} \cdot \cX := 0.\]
\end{definition}

By the Steenrod problem for orbifolds \ref{thm:intro-steenrod-problem}, there exists a basis $\{[\cX_i]\}$ of $H_*(\cO;\Q)$ consisting of the fundamental classes of embedded full suborbifolds $\cX_i \subset \cO$ whose normal bundles have fiberwise trivial isotropy action.

\begin{definition}\label{def:polyfold-invariant-intersection-number}
	Fix such a basis $\{[\cX_i]\}$ of $H_*(\cO;\Q)$.	
	We define the \textbf{polyfold invariant} as the homomorphism
		\[
		H_*(\cO; \Q) \to \Q,\qquad	\sum_i k_i [\cX_i] \mapsto \sum_i k_i \left( f|_{\cS} \cdot \cX_i \right)
		\]
	uniquely determined by evaluating the intersection number on the representing suborbifolds $\cX_i$ and linear extension.
\end{definition}

\subsection{Equivalence of the polyfold invariants}\label{subsec:equality-polyfold-invariants}

We now show that the polyfold invariants defined by the branched integral and by the intersection number are equivalent.
We also discuss the invariance of the intersection number.

\begin{proof}[Proof of Theorem~\ref{thm:equality-polyfold-invariants}]

We need to show that the branched integral the intersection number are related by the equation
	\[\int_{\cS} f^* \PD ([\cX]) = f|_{\cS} \cdot \cX.\]

When $\dim \cS + \dim \cX \neq \dim \cO$ the claim is clear, as both the intersection number and the branched integral are equal to zero.

Suppose that $\dim \cS + \dim \cX = \dim \cO$.
Let $n:= \dim \cS$.  By our assumptions, Lemma~\ref{lem:transverse-intersection} (\ref{lem:1}) implies that $f^{-1}(\cX)$ consists of a finite set of points, $\{[x_1],\ldots,[x_k]\}$.
For each point choose a representative $x_j \in S$ and let $U_j\subset Z$ be a $\bm{G}(x_j)$-invariant open neighborhood of $x_j$; it follows that $\cup_{j=1}^k \abs{U_j \cap S}$ is an open neighborhood of the points $\{[x_1],\ldots,[x_k]\}$.

By Proposition~\ref{prop:thom-class-poincare-dual} we may represent the de Rahm class $\PD([\cX]) \in H_{\dR}^n(\cO)$ as the Thom class of $\cX$.
We may assume that $\supp \PD([\cX]) \subset \cO$ is contained in an arbitrarily small neighborhood of $\cX$; hence by Lemma~\ref{lem:neighborhood} we may assume $\supp f^* \PD([\cX]) = \abs{\supp \hat{f}^* \PD ([\cX])} \subset \cup_{j=1}^k \abs{U_j \cap S}$.

By the definition of the branched integral, we may now write:
	\begin{align*}
	\int_{\cS} f^* \PD ([\cX])
		& = \sum_{j=1}^k \int_{\abs{U_j \cap S}} f^* \PD ([\cX]) \\
		& = \sum_{j=1}^k \left(\frac{1}{\sharp \bm{G}^\text{eff}(x_j)} \sum_{i\in I_j} w_i \int_{(M_i,o_i)} \hat{f}^* \PD ([\cX]) \right) \\
		& = \sum_{j=1}^k \left(\frac{1}{\sharp \bm{G}^\text{eff}(x_j)} \sum_{\{i\in I_j\mid x_j \in M_i \}} \sign (x_j; i) w_i \right) \\
		& = f|_{\cS} \cdot \cX.
	\end{align*}
The third equality follows by integrating the pullback of the Thom class on each branch, using the transversality assumption and the fact that $\PD([\cX])$ restricts to a generator of $H^n_{\dR,c}(N_{\hat{f}(x_j)} X)$ determined by the orientation of $NX$.
\end{proof}

\begin{remark}[Invariance of the intersection number]\label{rmk:invariance-intersection-number}
	We discuss the invariance of Definition~\ref{def:polyfold-invariant-intersection-number} from the choice of perturbation $p$ used to define a weighted branched suborbifold $\cS(p)$.
	
	To this end, consider a compact oriented weighted branched suborbifold $\cB$ with boundary $\partial \cB$.
	As usual, let $\cO$ be an oriented orbifold, let $\cX\subset \cO$ be a closed embedded full suborbifold whose normal bundle has fiberwise trivial isotropy action, and consider a smooth map $f: \cB \to \cO$.
	The arguments of Proposition~\ref{prop:transversality-perturbation-full-suborbifold} or Proposition~\ref{prop:transversality-regular-perturbation} show that without loss of generality we may assume
	\[
	f|_{\cB} \pitchfork \cX, \qquad f|_{\partial \cB} \pitchfork \cX.
	\]
	
	Suppose that $\dim \cB + \dim \cX = \dim \cO +1$. It can be shown that $f^{-1}(\cX)\subset \cB$ has the structure of a one-dimensional compact oriented weighted branched suborbifold with boundary $f^{-1}(\cX) \cap \partial \cB$.
	With this viewpoint, one necessarily expects the intersection number of $f$ restricted to the boundary to be zero,
	\[
	f|_{\partial \cB} \cdot \cX =0.
	\]
	Intuitively, this is because one should expect that a one-dimensional compact oriented weighted branched suborbifold can be decomposed into a (not necessarily disjoint) union of oriented weighted intervals $[0,1]$ and circles $S^1$. 
	From there, one would argue that the intersection number is equal to the signed weighted count of the boundary points of the intervals, and argue that this count is zero. 
	
	We do not pursue this approach, however.
	Instead, using Theorem~\ref{thm:equality-polyfold-invariants} we may observe that the homomorphism defined via the branched integral,
	\[
	H_*(\cO; \Q) \to \R, 
	\qquad \sum_i \alpha_i [\cX_i] \mapsto \sum_i \alpha_i \int_{\cS(p)} \hat{f}^* \PD([\cX_i]),
	\]
	and the homomorphism defined via the intersection number,
	\[
	H_*(\cO; \Q) \to \R,
	\qquad \sum_i \alpha_i [\cX_i] \mapsto \sum_i \alpha_i \left(	f|_{\cS(p)} \cdot \cX_i	\right),
	\]
	are identical.
	Since the first homomorphism does not depend on the choice of abstract perturbation nor on a choice of basis for $H_*(\cO;\Q)$, the same is true for the second homomorphism.
	
	We note that the second homomorphism is rationally valued and hence the branched integral is also rationally valued when evaluated on rational cohomology classes.
	This verifies a claim made in \cite[p.~12]{HWZGW}.
\end{remark}

\section*{Acknowledgment}
	This work is indebted to Katrin Wehrheim for their aid and insight in understanding the intricate machineries and untold depths of polyfold theory.


\end{document}